\documentclass{amsart}

\usepackage[margin=60pt]{geometry}
\usepackage{amsmath,amsthm,amssymb,bm}
\usepackage[colorlinks=true, pdfstartview=FitV, linkcolor=blue, citecolor=blue, urlcolor=blue]{hyperref}
\usepackage{cleveref}
\usepackage{mathtools}
\usepackage{graphicx,color}
\usepackage{ytableau}
\usepackage{tikz}
\usepackage[T1]{fontenc}
\usepackage{tgtermes}
\usepackage{todonotes}
\numberwithin{equation}{section}

\newtheorem{thm}{Theorem}[section]
\newtheorem{lem}[thm]{Lemma}
\newtheorem{prop}[thm]{Proposition}
\newtheorem{cor}[thm]{Corollary}
\theoremstyle{definition}
\newtheorem{exam}[thm]{Example}
\newtheorem{defn}[thm]{Definition}

\newtheorem{problem}[thm]{Problem}
\newtheorem{remark}[thm]{Remark}

\newcommand\shift{\operatorname{shift}}

\newcommand\Z{\mathbb{Z}}

\newcommand\fs{\operatorname{fs}}
\newcommand\bla{\boldsymbol\lambda}

\newcommand{\etalchar}[1]{$^{#1}$}

\newcommand{\boxes}[3][white]{
  \foreach \i in {0,...,\numexpr#3-1}{
    \draw[fill=#1, draw = black] (#2,\i) rectangle ++(1,1);
  }
}

\newcommand{\N}{\mathbb{N}}
\newcommand\ppm{\mathsf{pm}}

\newcommand\qand{\quad\text{and}\quad}
\newcommand\Qbinom[3]{\genfrac{[}{]}{0pt}{}{#1}{#2}_{#3}}

\definecolor{darkblue}{rgb}{0.0,0,0.7}
\renewcommand\emph[1]{\textcolor{darkblue}{\it #1}}

\title{Andrews--Gordon type identities with parity restrictions through particle motion}

\author{Jehanne Dousse}
\address{Université de Genève, 7--9, rue Conseil Général, 1205 Genève, Switzerland}
\email{jehanne.dousse@unige.ch}

\author{Jihyeug Jang}
\address{Université de Genève, 7--9, rue Conseil Général, 1205 Genève, Switzerland}
\email{jihyeug.jang@unige.ch}

\begin{document}

\maketitle
% \tableofcontents

\begin{abstract}
  In this paper, we use the particle motion bijection introduced by
  Warnaar and developed by the two authors, Jouhet and Konan, to study
  q-series and partition identities of the Andrews--Gordon type with
  parity restrictions. These restrictions are of the type ``even
  (resp. odd) parts appear an even number of times". We prove
  $q$-series identities where a multisum equals a sum of products,
  which generalise identities of Andrews and Kim--Yee in a similar way
  that Stanton's identities generalised the Andrews--Gordon
  identities. As a consequence of our results, we obtain a simple
  proof of a recent identity of Chern--Li--Stanton--Xue--Yee related
  to Ariki--Koike algebras.

\end{abstract}

\section{Introduction}
A \emph{partition} of a positive integer \( n \) is a non-increasing
sequence \(\lambda= (\lambda_1, \lambda_2, \dots, \lambda_\ell)\) of
positive integers whose sum is \(n\). The integer \(n\) is called the
\emph{weight} of the partition \(\lambda\) and is also denoted $|\lambda|$, and \(\ell\) is called its
\emph{length}. Every partition \( \lambda \) can also be written in
terms of its \emph{frequency sequence} \( f = (f_1, f_2, \cdots ) \),
where \( f_j \) denotes the number of occurrences of the part \( j \)
in \( \lambda \). Its weight is
$|f| = \sum_{k \geq 1} k f_k.$
In this paper, we will often write partitions in
terms of their frequency sequences and represent them graphically with
columns of boxes on the \(x\)-axis.

\begin{exam}
The partition \( \lambda=(6,5,5,5,3,3,1,1,1,1)\) has frequency sequence \(f=(4,0,2,0,3,1,0,0, \dots)\) which can be represented by the picture below.
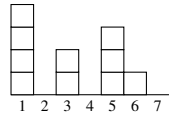
\center
\begin{tikzpicture}[scale = .3]
        \draw (1,0) -- (8,0);
        \foreach \x in {1,...,7}
      \node at (\x+ .5, -0.5) {\tiny \number\numexpr\x \relax};
        \boxes{1}{4}   
        \boxes{3}{2}   
        \boxes{5}{3}   
        \boxes{6}{1}   
\end{tikzpicture}
\end{exam}

Using the \emph{\( q \)-Pochhammer symbols} (see~\cite{GR})
\begin{align*}
(a;q)_n &:= (a)_n = \prod_{k=0}^{n-1} (1-aq^k),\\
(a_1,\ldots,a_m;q)_n &:=(a_1;q)_n\cdots(a_m;q)_n,
\end{align*}
for $ n \in  \N \cup \{+\infty\}$, the \emph{Rogers--Ramanujan identities}~\cite{RR19} state that for \( a \in \{0,1\} \),
\begin{equation}\label{eq:RR}
  \sum_{n \geq 0} \frac{q^{n^2+ (1-a)n}}{(q)_n} = \frac{1}{(q^{2-a},q^{3+a};q^5)_{\infty}}.
\end{equation}
As combinatorial identities, they assert that for \( a \in \{0,1\} \)
and for all $n \in \N$, the number of partitions
\( \lambda = (\lambda_1, \lambda_2, \dots, \lambda_\ell)\) of \( n \)
such that \( \lambda_i - \lambda_{i+1} \geq 2 \) for all \( i \),
where the part \( 1 \) appears at most \( a \) times, is equal to the
number of partitions of \( n \) into parts congruent to
\( \pm (2-a) \) modulo \( 5 \). Equivalently, in terms of frequency
sequences, they can be written as follows: for \( a \in \{0,1\} \) and
for all \(n \in \N\), the number of partitions \(f\) of \(n\) such
that \( f_i + f_{i+1} \leq 1 \) for all \( i \) and \( f_1 \leq a \),
equals the number of partitions of \(n\) into parts congruent to
\( \pm (2-a) \) modulo \( 5 \).

The Rogers--Ramanujan identities are undoubtedly among the most famous
partition/\(q\)-series identities and appear in several fields of
mathematics: they revealed the role of vertex operator algebras in
representation theory \cite{Lepowsky,Lepowsky2}, arise naturally in
the theory of arc spaces in algebraic geometry \cite{Mourtada}, and
played a key role in the exact solution of the hard hexagonal model in
mathematical physics \cite{Baxter}, to name only a few examples. For
more detail about these identities, see the book \cite{Sills}.

The Rogers--Ramanujan identities were extended combinatorially by
Gordon~\cite{Go}.
\begin{thm}[Gordon's identities~\cite{Go}]\label{thm:Gordon}
For integers
\( k \geq 1 \) and \( 0 \leq r \leq k \),  the number of
partitions \( \lambda =  (\lambda_1, \lambda_2,  \dots, \lambda_\ell)\) of \( n \) such that
\( \lambda_i - \lambda_{i+k} \geq 2 \) for all \( i \), and where the
part \( 1 \) appears at most \( (k-r) \) times, is equal to the
number of partitions of \( n \) into parts not congruent to
\( 0, \pm (k-r+1) \) modulo \( 2k+3 \).
\end{thm}
Equivalently, the difference condition above can be rephrased in terms of frequencies as follows:
\begin{equation}\label{eq:comb_AG}
  \text{partitions of $n$ whose frequency sequence \( (f_i)_{i \geq 1} \) satisfies
    \( f_i + f_{i+1} \leq k \) for all \( i \), and
    \( f_1 \leq k-r \).}
\end{equation}

Andrews~\cite{AndrewsGordon} then found a $q$-series counterpart to the Gordon identities, obtaining what is now known as the Andrews--Gordon identities. 

\begin{thm}[Andrews--Gordon identities~\cite{AndrewsGordon}]\label{thm:AG}
Let \(k \geq 1\) and \(0 \leq r \leq k\) be two integers. Then
\begin{equation}\label{eq:AG}
\sum_{s_1\geq\dots\geq s_{k}\geq0}\frac{q^{s_1^2+\dots+s_{k}^2+s_{k-r+1}+\dots+s_{k}}}{(q)_{s_1-s_2}\dots(q)_{s_{k-1}-s_{k}}(q)_{s_{k}}}=\frac{(q^{k+1-r},q^{k+2+r},q^{2k+3};q^{2k+3})_\infty}{(q)_\infty}.
\end{equation}
\end{thm}

While it is clear that the right-hand side of \eqref{eq:AG} is the generating function for partitions into parts not congruent to \( 0, \pm (k-r+1) \) modulo \( 2k+3 \), it is much harder to prove that the left-hand side is the generating function for partitions with difference/frequency conditions. This fact was originally proved by Andrews \cite{AndrewsGordon} using recurrences, and it took more than twenty years until a bijective proof was found by Warnaar \cite{WarnaarParticle}. Note that this bijection only shows the correspondence between the sum side of \eqref{eq:AG} and the frequency conditions; finding a bijective proof of the Andrews--Gordon identities themselves is still an open problem.

Warnaar's bijection relies on \emph{particle motion} on frequency
sequence diagrams starting from a minimal partition satisfying the
difference conditions. It was applied on usual frequency sequences
\( f = (f_1, f_2, \cdots ) \). In \cite{DJK2024}, the first author,
Jouhet and Konan generalised Warnaar's approach by allowing parts of
size $0$, hence considering frequency sequences of the form
\( f = (f_0, f_1, \cdots ) \). For the purposes of this paper, we need
to consider parts of size $0$ and $-1$. For simplicity and generality,
we describe the theory in an even more general context by applying
particle motion to generalised frequency sequences allowing parts of
any size. So from now on, a \emph{partition} will be a finite
non-increasing sequence of \emph{non-negative} integers, and a
\emph{generalised frequency sequence} will be a sequence
$(f_i)_{i \in \Z}$ such that only finitely many of the $f_u$'s are
non-zero. The principle of Warnaar's particle motion and its
generalisations is described in Section \ref{sec:pm}, together with
applications. For instance, in \cite{DJJ25}, the two authors and
Jouhet used their generalisation of particle motion to prove this
theorem of Stanton, which generalises Theorem \ref{thm:AG}.

\begin{thm}\cite[Theorem~3.2]{Stanton2018}\label{thm:Sta3.2}
  Let \( j,r \geq 0 \) and \( k \geq 1 \) be integers such that
  \( j+r \leq k \). Then
\begin{equation}\label{eq:sAG}
\sum_{s_1\geq\dots\geq s_{k}\geq0}\frac{q^{s_1^2+\dots+s_{k}^2-s_1-\dots-s_j+s_{k-r+1}+\dots+s_{k}}}{(q)_{s_1-s_2}\dots(q)_{s_{k-1}-s_{k}}(q)_{s_{k}}}=\sum_{s=0}^j\frac{(q^{k+1-r+j-2s},q^{k+2+r-j+2s},q^{2k+3};q^{2k+3})_\infty}{(q)_\infty}.
\end{equation}
\end{thm}

\medskip

In this paper, our goal is to use particle motion to prove $q$-series
identities related to partitions with parity conditions.

The study of partition identities with the frequency conditions
\eqref{eq:comb_AG} but with added parity conditions originated from
the work of Andrews \cite{Andrews2010}, and was followed by several
others since \cite{Kursungoz2010, KY2013, KL2024, CLSXY2024}. Let us
first describe Andrews' result, slightly reformulated to fit the
context of generalised frequency sequences.

\begin{defn}
\label{def:W}
Let \( \mathcal{W}_{k,a}\) (resp. \( \overline{\mathcal{W}}_{k,a} \)) denote the
set of generalised frequency sequences \( (f_i)_{i \in \Z} \) such that
\begin{itemize}
\item  $f_i + f_{i+1} \leq k$ for all \( i > 0 \),
\item  $f_i=0$ for all \( i \leq 0 \),
\item $f_1 \leq a$,
\item $f_i$ is even if $i$ is even (resp. odd). That is, even (resp. odd) parts appear an even number of times.
\end{itemize}
\end{defn}

Andrews proved the following two theorems using recurrence equations.
\begin{thm}[Andrews \cite{Andrews2010}]
\label{thm:1}
For \( k \geq a \geq 0 \) and \( k \equiv a \pmod{2}  \),
    \begin{multline*}
 \sum_{f \in\mathcal{W}_{k,a}} q^{|f|}
      =\sum_{s_1 \ge \cdots \ge s_{k} \geq 0}
      \frac{q^{s_1^2 + \cdots + s_{k}^2 + 2s_{a+1} + 2s_{a+3} + \cdots + 2s_{k-1}}}
      {(q^2;q^2)_{s_1 - s_2} \cdots (q^2;q^2)_{s_{k-1} - s_{k}}(q^2;q^2)_{s_{k}}}\\
      = \frac{(-q;q^2)_\infty(q^{a+1}, q^{2k+3-a}, q^{2k+4}; q^{2k+4})_\infty}{(q^2;q^2)_\infty}.
    \end{multline*}
  
\end{thm}  

\begin{thm}[Andrews \cite{Andrews2010}]
\label{thm:3}
For \( k \geq a \geq 1 \), \( k \) even and \( a \) odd,
    \begin{multline*}
 \sum_{f \in \overline{\mathcal{W}}_{k,a}} q^{|f|}
      = \sum_{s_1 \ge \cdots \ge s_{k} \geq 0}
      \frac{q^{s_1^2 + \cdots + s_{k}^2 + (s_1 - s_2 + s_3 - \cdots + s_{a-2} - s_{a-1}) + (s_{a} + s_{a+1} + \cdots + s_{k})}}
      {(q^2;q^2)_{s_1 - s_2} \cdots (q^2;q^2)_{s_{k-1} - s_{k}}(q^2;q^2)_{s_{k}}}\\
      = \frac{(-q^2;q^2)_\infty(q^{a+1}, q^{2k+3-a}, q^{2k+4}; q^{2k+4})_\infty}{(q^2;q^2)_\infty}.
    \end{multline*}
\end{thm}

Then, using Gordon marking, Kur{\c{s}}ung{\"o}z \cite{Kursungoz2010}
was able to give a formula for the generating function of
\(\mathcal{W}_{k,a}(n)\) for \( k \not\equiv a \pmod{2}\) and
\(\overline{\mathcal{W}}_{k,a}(n)\) for \( k \) odd and \( a \) even
as a sum. Note that in these particular cases, as well as for the sum
side of the Andrews--Gordon identities, Gordon marking is equivalent
to particle motion. Then Kim and Yee \cite{KY2013} reproved these sum
sides and were able to give product (or sum of products) expressions
as well, using recurrences and \(q\)-series identities. Their results
are the following.

\begin{thm}[Kim--Yee \cite{KY2013}, sum side by Kur{\c{s}}ung{\"o}z \cite{Kursungoz2010}]
\label{thm:2}
For \( k \geq a \geq 0 \) and \( k \not\equiv a \pmod{2}  \),
    \begin{multline*}
 \sum_{f \in\mathcal{W}_{k,a}} q^{|f|}
      =\sum_{s_1 \ge \cdots \ge s_{k} \geq 0}
      \frac{q^{s_1^2 + \cdots + s_{k}^2 + 2s_{a+1} + 2s_{a+3} + \cdots + 2s_{k}}}
      {(q^2;q^2)_{s_1 - s_2} \cdots (q^2;q^2)_{s_{k-1} - s_{k}}(q^2;q^2)_{s_{k}}}\\
      = \frac{(-q^3;q^2)_\infty(q^{a+2}, q^{2k+2-a}, q^{2k+4}; q^{2k+4})_\infty}{(q^2;q^2)_\infty} + q \cdot \frac{(-q^3;q^2)_\infty(q^{a}, q^{2k+4-a}, q^{2k+4}; q^{2k+4})_\infty}{(q^2;q^2)_\infty}.
    \end{multline*}
\end{thm}

\begin{thm}[Kim--Yee \cite{KY2013}, sum side by Kur{\c{s}}ung{\"o}z \cite{Kursungoz2010}]
\label{thm:4}
For \( k \geq a \geq 0 \), \( k \) odd and \( a \) even,
    \begin{multline*}
 \sum_{f \in \overline{\mathcal{W}}_{k,a}} q^{|f|}
      = \sum_{s_1 \ge \cdots \ge s_{k} \geq 0}
      \frac{q^{s_1^2 + \cdots + s_{k}^2 + (s_1 - s_2 + s_3 - \cdots + s_{a-1} - s_{a}) + (s_{a+1} + \cdots + s_{k})}}
      {(q^2;q^2)_{s_1 - s_2} \cdots (q^2;q^2)_{s_{k-1} - s_{k}}(q^2;q^2)_{s_{k}}}\\
      = \frac{(-q^2;q^2)_\infty(q^{a+2}, q^{2k+2-a}, q^{2k+4}; q^{2k+4})_\infty}{(q^2;q^2)_\infty}.
    \end{multline*}
     
\end{thm}

Note that for \( k \geq a \geq 1 \) and \( a \) odd,
\( \overline{\mathcal{W}}_{k,a} = \overline{\mathcal{W}}_{k,a-1} \).
The cases where \( k \equiv a \pmod{2} \) thus follow immediately from
Theorems \ref{thm:3} and \ref{thm:4}. That is, if \( a \) and \( k \)
are even, then
    \begin{multline*}
     \sum_{f \in \overline{\mathcal{W}}_{k,a}} q^{|f|}
      = \sum_{f \in \overline{\mathcal{W}}_{k,a+1}} q^{|f|}
      = \sum_{s_1 \ge \cdots \ge s_{k} \geq 0}
      \frac{q^{s_1^2 + \cdots + s_{k}^2 + (s_1 - s_2 + s_3 - \cdots + s_{a-1} - s_{a}) + (s_{a+1} + \cdots + s_{k})}}
      {(q^2;q^2)_{s_1 - s_2} \cdots (q^2;q^2)_{s_{k-1} - s_{k}}(q^2;q^2)_{s_{k}}}\\
      = \frac{(-q^2;q^2)_\infty(q^{a+2}, q^{2k+2-a}, q^{2k+4}; q^{2k+4})_\infty}{(q^2;q^2)_\infty}.
    \end{multline*}
    If \( a \) and \( k \) are odd, then
    \begin{multline*}
  \sum_{f \in \overline{\mathcal{W}}_{k,a}} q^{|f|}
      = \sum_{f \in \overline{\mathcal{W}}_{k,a-1}} q^{|f|}
      = \sum_{s_1 \ge \cdots \ge s_{k} \geq 0}
      \frac{q^{s_1^2 + \cdots + s_{k}^2 + (s_1 - s_2 + s_3 - \cdots + s_{a-2} - s_{a-1}) + (s_{a} + \cdots + s_{k})}}
      {(q^2;q^2)_{s_1 - s_2} \cdots (q^2;q^2)_{s_{k-1} - s_{k}}(q^2;q^2)_{s_{k}}}\\
      = \frac{(-q^2;q^2)_\infty(q^{a+1}, q^{2k+3-a}, q^{2k+4}; q^{2k+4})_\infty}{(q^2;q^2)_\infty}.
    \end{multline*}

Thus all the above results can be unified as follows.
\begin{thm}[Unification of Theorems~\ref{thm:1}--\ref{thm:4}]
\label{thm:unif}
For all integers $k \geq a \geq 0$,

\begin{multline}\label{eq:eq:W}
 \sum_{f \in\mathcal{W}_{k,a}} q^{|f|}
  =\sum_{s_1 \ge \cdots \ge s_{k} \geq 0}
  \frac{q^{s_1^2 + \cdots + s_{k}^2 + 2s_{a+1} + 2s_{a+3} + \cdots + (1-(-1)^{k+a})s_{k}}}
  {(q^2;q^2)_{s_1 - s_2} \cdots (q^2;q^2)_{s_{k-1} - s_{k}}(q^2;q^2)_{s_{k}}}\\
  = \frac{(-q^3;q^2)_\infty}{(q^2;q^2)_\infty}
  \left((q^{k-2 \lfloor \frac{k-a}{2} \rfloor +1}, q^{k+2 \lfloor \frac{k-a}{2} \rfloor +3}, q^{2k+4}; q^{2k+4})_\infty + q (q^{k-2 \lfloor \frac{k-a+1}{2} \rfloor +1}, q^{k+2 \lfloor \frac{k-a+1}{2} \rfloor +3}, q^{2k+4}; q^{2k+4})_\infty\right),
\end{multline}
and
\begin{multline}\label{eq:W_bar}
\sum_{f \in \overline{\mathcal{W}}_{k,a}} q^{|f|}
  = \sum_{s_1 \ge \cdots \ge s_{k} \geq 0}
  \frac{q^{s_1^2 + \cdots + s_{k}^2 + (s_1 - s_2 + s_3 - \cdots + (-1)^{a+1} s_{a}) + (s_{a+1} + \cdots + s_{k})}}
  {(q^2;q^2)_{s_1 - s_2} \cdots (q^2;q^2)_{s_{k-1} - s_{k}}(q^2;q^2)_{s_{k}}}\\
  = \frac{(-q^2;q^2)_\infty(q^{2 \lfloor \frac{a+2}{2} \rfloor}, q^{2k+4-2 \lfloor \frac{a+2}{2} \rfloor}, q^{2k+4}; q^{2k+4})_\infty}{(q^2;q^2)_\infty}.
\end{multline}
\end{thm}

Our main purpose in this paper is to prove the following identities,
which generalise Theorem \ref{thm:unif} in the same way that Stanton's
\Cref{thm:Sta3.2} generalises the Andrews--Gordon identities. We
achieve this using particle motion.

\begin{thm}\label{thm:main}
For non-negative integers \( j,r \), and \( k \) with
  \( j+r \leq k \), let \( \mathcal{Z}^o_{j,r,k} \) be the set
  of all frequency sequences \( (f_i)_{i \in \Z} \) such that
  \begin{itemize}
  \item \(f_i + f_{i+1} \leq k \) for all  \( i \geq 0 \),
  \item \(f_i=0\) for all \(i<0\),
  \item \( f_0 \leq j \) and \( 2f_0 + f_1 \leq k-r+j \), 
  \item \( f_i \) is even if \( i \) is odd. That is, odd parts appear an even number of times.
  \end{itemize}
  Then
  \begin{multline*}
   \sum_{ f \in \mathcal{Z}^o_{j,r,k} } q^{|f|} = \sum_{s_1 \geq \cdots \geq s_k \geq 0}\frac{q^{s_1^2 + \cdots + s_k^2 - (s_1 + \cdots + s_j) + (s_{j+1} - s_{j+2} + s_{j+3} - \cdots \pm s_{k-r}) + (s_{k-r+1} + \cdots + s_k)}}{(q^2;q^2)_{s_1 - s_2} \cdots (q^2;q^2)_{s_{k-1} - s_k} (q^2;q^2)_{s_k}} \\
    = \sum_{s=0}^j  \frac{(-q^2;q^2)_\infty(q^{ 2 \lfloor \frac{k+2-r+j-2s}{2} \rfloor}, q^{2k+4 - 2 \lfloor \frac{k+2-r+j-2s}{2} \rfloor} , q^{2k+4}; q^{2k+4})_\infty}{(q^2;q^2)_\infty}.
  \end{multline*}
\end{thm}
When \(j=0\), Theorem \ref{thm:main} reduces to \eqref{eq:W_bar} after setting \(r=k-a\).
  
  \begin{thm}\label{thm:even1}
For non-negative integers \( a,b \), and \( k \) with
  \( 2a + 2b \leq k \), let \( \mathcal{Z}^e_{a,b,k} \) be the set
  of all frequency sequences \( (f_i)_{i \in \Z} \) such that
  \begin{itemize}
  \item  \( f_i + f_{i+1} \leq k \) for all  \( i \geq 0 \),
  \item \(f_i=0\) for all \(i<0\),
  \item \( f_0 \leq 2a \) and \( 2f_0 + f_1 \leq k-2b+2a \), 
  \item \( f_i \) is even if \( i \) is even. That is, even parts appear an even number of times. 
  \end{itemize} 
Then
  \begin{multline*}
 \sum_{ f \in \mathcal{Z}^e_{a,b,k} } q^{|f|} =   \sum_{ s_1 \geq \dots \geq s_k \geq 0 } \frac{q^{s_1^2 + \cdots + s_k^2 -2(s_2 + s_4 + \cdots + s_{2a}) +2(s_{k-2b+1} + s_{k-2b+3} + \cdots + s_{k-1})}}{(q^2;q^2)_{s_1 - s_2} \cdots (q^2;q^2)_{s_{k-1} - s_k}(q^2;q^2)_{s_k}} \\
   = \sum_{ s = 0 }^{ a } \frac{(-q;q^2)_\infty (q^{k+1+2a-2b-4s}, q^{k+3-2a+2b+4s}, q^{2k+4}; q^{2k+4})_\infty}{(q^2;q^2)_\infty}.
  \end{multline*}
\end{thm}

\begin{thm}\label{thm:even2}
For non-negative integers \( a,b \), and \( k \) with
  \( 2a + 2b -1 \leq k \), let
  \( \widetilde{\mathcal{Z}}^e_{a,b,k} \)  be the set of all
  frequency sequences \( (f_i)_{i \in \Z} \) such that
  \begin{itemize}
  \item \( f_i + f_{i+1} \leq k \) for all \( i \geq 0 \),
  \item \(f_i=0\) for all \(i<0\),
  \item \( f_0 \leq 2a \) and \( 2f_0 + f_1 \leq k-2b+2a+1 \), 
  \item \( f_i \) is even if \( i \) is even. That is, even parts appear an even number of times.
  \end{itemize}
Then
  \begin{multline*}
   \sum_{f \in \widetilde{\mathcal{Z}}_{a,b,k}^e}q^{|f|}=  \sum_{ s_1 \geq \dots \geq s_k \geq 0 } \frac{q^{s_1^2 + \cdots + s_k^2 -2(s_2 + s_4 + \cdots + s_{2a}) +2(s_{k-2b+2} + s_{k-2b+4} + \cdots + s_{k})}}{(q^2;q^2)_{s_1 - s_2} \cdots (q^2;q^2)_{s_{k-1} - s_k}(q^2;q^2)_{s_k}} \\
   = \frac{(-q^3;q^2)_\infty}{(q^2;q^2)_\infty}\sum_{ s = 0 }^{ a } \Bigg[ (q^{k+3+2a-2b-4s}, q^{k+1-2a+2b+4s}, q^{2k+4}; q^{2k+4})_\infty + q (q^{k+1+2a-2b-4s}, q^{k+3-2a+2b+4s}, q^{2k+4}; q^{2k+4})_\infty \Bigg].
  \end{multline*}
\end{thm}
When \(a=0\), Theorems \ref{thm:even1} and \ref{thm:even2} reduce to Theorems \ref{thm:1} and \ref{thm:2}, respectively.

Other nice corollaries follow from them as well. Indeed, \Cref{thm:main} implies the following.
\begin{cor}\label{cor:odd}
  For integers \( k \) and \( a \) with \( k \geq 1 \) and
  \( 0 \leq a \leq k \),
  \begin{multline*}
    \sum_{s_1 \ge \cdots \ge s_{k} \geq 0}
    \frac{q^{s_1^2 + \cdots + s_{k}^2 - (s_1 + \cdots + s_a) + (s_{a+1} - s_{a+2} + \cdots +(-1)^{k-a-1} s_{k})}}
    {(q^2;q^2)_{s_1 - s_2} \cdots (q^2;q^2)_{s_{k-1} - s_{k}}(q^2;q^2)_{s_{k}}}\\
    = \sum_{i=0}^a  \frac{(-q^2;q^2)_\infty(q^{ 2 \lfloor \frac{k-i+2}{2} \rfloor}, q^{2k+4 - 2 \lfloor \frac{k-i+2}{2} \rfloor} , q^{2k+4}; q^{2k+4})_\infty}{(q^2;q^2)_\infty}.
  \end{multline*}
\end{cor}
And Theorems \ref{thm:even1} and \ref{thm:even2} imply the following.
\begin{cor}\label{cor:even}
  For integers \( k \) and \( a \) with \( k \geq 1 \) and \( 0 \leq 2a \leq k \),
  \[
    \sum_{s_1 \ge \cdots \ge s_{k} \geq 0}
    \frac{q^{s_1^2 + \cdots + s_{k}^2 - 2(s_2 + s_4 + \cdots + s_{2a})}}
    {(q^2;q^2)_{s_1 - s_2} \cdots (q^2;q^2)_{s_{k-1} - s_{k}}(q^2;q^2)_{s_{k}}}
    = \sum_{i=0}^{a} \frac{(-q;q^2)_\infty(q^{k+1-2i}, q^{k+3 +2i} , q^{2k+4}; q^{2k+4})_\infty}{(q^2;q^2)_\infty}.
  \]
\end{cor}

\medskip

The interest for these theorems with parity restrictions became even greater recently, as it was discovered that they have some connection to the representation theory of Ariki--Koike algebras, introduced independently by Ariki--Koike \cite{AK} and Brou\'e and Malle \cite{BM}.
Ariki and Mathas \cite{AM} showed that simple modules of Ariki--Koike algebras are labelled by Kleshchev multipartitions, and that their generating function is
\begin{equation}
\label{eq:AKprod}
\frac{(q^{a+1}, q^{k+1-a}, q^{k+2} ; q^{k+2})_\infty}{(q)_\infty (q;q^2)_\infty},
\end{equation}
which is exactly the product side of Theorem \ref{thm:4} where \(q\) is replaced by \(q^{1/2}\).
Interestingly, this product is also the principal specialisation of the character of standard modules of level $k$ for the affine Kac--Moody Lie algebra $A_1^{(1)}$, as studied by Meurman and Primc \cite{MP} and the first author, Hardiman and Konan \cite{DHK}.

By studying the connection between \eqref{eq:AKprod} and Ariki--Koike algebras, Chern, Li, Stanton, Xue, and Yee \cite{CLSXY2024} discovered and proved the following \(q\)-series identity, which was later reproved by Kanade and Lovejoy \cite{KL2024} in a more simple manner using the machinery of Bailey pairs.

\begin{thm}[\cite{KL2024, CLSXY2024}]\label{thm:AK}
  Let \( k \geq 1 \) and \( 0 \leq a \leq k-1 \).
  Then 
  \begin{equation}\label{eq:AK_id}
  \begin{split}
    \frac{(q^{a+1}, q^{k+1-a}, q^{k+2} ; q^{k+2})_\infty}{(q)_\infty (q;q^2)_\infty}
    &= \sum_{n_k,\dots,n_1 \geq 0} \frac{q^{\binom{n_k +1}{2} + \cdots + \binom{n_1 +1}{2}}}{(q)_{n_k}} \prod_{i=1}^{k-1} \Qbinom{n_{i+1} + \delta_{a,i}}{n_i}{q} \\
    &= \sum_{
    \substack{
    N_k \geq N_{k-1} \geq \cdots \geq N_{a+1} \\
    N_{a+1}+1 \geq N_a \geq \cdots \geq N_2 \geq 0
    }
    } \frac{q^{\sum_{i=2}^k \binom{N_i +1}{2}}(-q ; q)_{N_{2+\delta_{a+1,2}}}\left(1-q^{N_{a+1}+1}\right) }{\left(1-q^{N_{a+1}-N_a+1}\right) \prod_{i=2}^k(q ; q)_{N_i-N_{i-1}}},
  \end{split}
  \end{equation}
where \( N_1 = 0 \) and \( N_{k+1} = \infty \).
\end{thm}
Here we used the \(q\)-binomial coefficients, defined for all \(0\leq j \leq n\) by
\[\left[{n\atop j}\right]_q:=\frac{(q)_n}{(q)_j(q)_{n-j}}.\]

Particle motion provides an even simpler proof of Theorem \ref{thm:AK}, as it follows directly from Theorem \ref{thm:main}.

Indeed, setting  \(s_i =n_{k+1-i}\) for \(i \in \{1, \dots, a \}\) and \(s_i =n_{k+1-i}+1\) for \(i \in \{a+1, \dots, k\}\), \Cref{eq:AK_id} can be rewritten as
\begin{equation}\label{eq:modified_AK}
  \sum_{s_1 \ge \cdots \ge s_k \geq 0} \frac{q^{(s_1^2 + \cdots + s_k^2 - s_1 - \cdots - s_{k-a} + s_{k-a+1} + \cdots + s_k)/2}(1-q^{s_{k-a}})}{(q)_{s_1 - s_2} \cdots (q)_{s_{k-1} - s_k}(q)_{s_k}}
  = \frac{\left(q^{a+1} , q^{k+1-a} , q^{k+2} ; q^{k+2}\right)_{\infty}}{(q ; q)_{\infty}\left(q ; q^2\right)_{\infty}}.
\end{equation}
%Replacing $q$ by $q^2$, we have
%\[
%  \sum_{s_1 \ge \cdots \ge s_k \geq 0} \frac{q^{s_1^2 + \cdots + s_k^2 - s_1 - \cdots - s_{k-a} + s_{k-a+1} + \cdots + s_k}(1-q^{2s_{k-a}})}{(q^2;q^2)_{s_1 - s_2} \cdots (q^2;q^2)_{s_{k-1} - s_k}(q^2;q^2)_{s_k}}
%  = \frac{\left(q^{2a+2} , q^{2k+2-2a} , q^{2k+4} ; q^{2k+4}\right)_{\infty}}{(q^2 ; q^2)_{\infty}\left(q^2 ; q^4\right)_{\infty}}.
%\]

Define
\[
  \operatorname{AK}_{a,k}(q) := \sum_{s_1 \ge \cdots \ge s_k \geq 0} \frac{q^{s_1^2 + \cdots + s_k^2 - s_1 - \cdots - s_{k-a} + s_{k-a+1} + \cdots + s_k}}{(q^2;q^2)_{s_1 - s_2} \cdots (q^2;q^2)_{s_{k-1} - s_k}(q^2;q^2)_{s_k}}.
\]

Setting \( j = k-a \) and \( r = a \) in \Cref{thm:main} yields the
following.

\begin{cor}\label{cor:2}
  \[
    \operatorname{AK}_{a,k}(q)
    = \sum_{s_1 \ge \cdots \ge s_k \geq 0} \frac{q^{s_1^2 + \cdots + s_k^2 - s_1 - \cdots - s_{k-a} + s_{k-a+1} + \cdots + s_k}}{(q^2;q^2)_{s_1 - s_2} \cdots (q^2;q^2)_{s_{k-1} - s_k}(q^2;q^2)_{s_k}}
    = \sum_{i = a}^{k} \frac{(q^{2i+2},q^{2k+2-2i},q^{2k+4}; q^{2k+4})_\infty}{(q^2;q^2)_\infty(q^2;q^4)_\infty}.
  \]
\end{cor}
%\begin{proof}
%  By \Cref{thm:main} with \( j = k-a \) and \( r = a \), we obtain
%  \begin{align*}
%    \sum_{s_1 \ge \cdots \ge s_k \geq 0} \frac{q^{s_1^2 + \cdots + s_k^2 - s_1 - \cdots - s_{k-a} + s_{k-a+1} + \cdots + s_k}}{(q^2;q^2)_{s_1 - s_2} \cdots (q^2;q^2)_{s_{k-1} - s_k}(q^2;q^2)_{s_k}}
%    &= \sum_{s = 0}^{k-a} \frac{(-q^2;q^2)_\infty (q^{2k+2-2a-2s},q^{2a+2a+2},q^{2k+4}; q^{2k+4})_\infty}{(q^2;q^2)_\infty} \\
%    &= \sum_{i = a}^{k} \frac{(q^{2i+2},q^{2k+2-2i},q^{2k+4}; q^{2k+4})_\infty}{(q^2;q^2)_\infty(q^2;q^4)_\infty}.
%  \end{align*}
%\end{proof}

Theorem~\ref{thm:AK} (in the form \eqref{eq:modified_AK}) then follows
directly by taking difference
\( \operatorname{AK}_{a,k}(q) - \operatorname{AK}_{a+1,k}(q) \) in
\Cref{cor:2} and replacing \( q \) by \( q^{1/2} \).

\begin{remark}
This is reminiscent of a proof strategy in \cite{ADJM} where the identity
\[
  \sum_{s_1\geq\dots\geq
    s_{r-1}\geq0}\frac{q^{s_1^2+\dots+s_{r-1}^2-s_1-\dots-s_i}(1-q^{s_i})}{(q)_{s_1-s_2}\dots(q)_{s_{r-2}-s_{r-1}}(q)_{s_{r-1}}}=\frac{(q^{r-i},q^{r+i+1},q^{2r+1};q^{2r+1})_\infty}{(q)_\infty},
\]
valid for all integers \(r>0\) and \(0\leq i\leq r-1\), was proved by taking the difference of the cases \(i\) and \(i-1\) in this identity of Bressoud ~\cite[(3.3)]{Br80}:
\[
  \sum_{s_1\geq\dots\geq
    s_{r-1}\geq0}\frac{q^{s_1^2+\dots+s_{r-1}^2-s_1-\dots-s_i}}{(q)_{s_1-s_2}\dots(q)_{s_{r-2}-s_{r-1}}(q)_{s_{r-1}}}=\sum_{k=0}^{i}\frac{(q^{r-i+k},q^{r+i-k+1},q^{2r+1};q^{2r+1})_\infty}{(q)_\infty}.
\]
\end{remark}

\medskip

The paper is organised as follows. In Section \ref{sec:pm}, we introduce particle motion and prove some of its properties, including some related to parity. In Sections \ref{sec:proofmain}, \ref{sec:general-case-2} and \ref{sec:last}, we use particle motion and some enumeration techniques to prove Theorems \ref{thm:main}, \ref{thm:even1} and \ref{thm:even2}, respectively, together with their corollaries. We conclude with some open problems.

\section{Particle motion and its properties}
\label{sec:pm}
We describe the principle of Warnaar's particle motion adapted to the context of generalised frame sequences, following the formulation of \cite{DJJ25}.

Let \(f=(f_i)_{i \in \Z}\) be a generalised frequency sequence and
\(u\) be an index such that, letting \(h:= f_u +f_{u+1}\),
\[
  f_v +f_{v+1} \leq h \text{ for all } v \geq u.
\]
Then the process applied at index \(u\) depends on whether
\(f_{u+1} +f_{u+2}<h\) or \(f_{u+1} +f_{u+2}=h\):
\begin{enumerate}
\item If \(f_{u+1} +f_{u+2}<h\), then transform \((f_u, f_{u+1})\) to
  \((f_u-1,f_{u+1}+1)\), leave \(u\) unchanged (\emph{particle
    motion}).
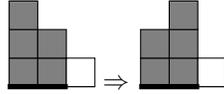
\[
\begin{tikzpicture}[scale = .38]
          \draw (1,0) -- (4,0);
          \boxes[gray]{1}{3}   
          \boxes[gray]{2}{2}   
          \boxes{3}{1}    
          \draw[line width = 2pt] (1- .05,0) -- (3 + .05,0);
        \end{tikzpicture}
        \Rightarrow
        \begin{tikzpicture}[scale = .38]
          \draw (1,0) -- (4,0);
          \boxes[gray]{1}{2}   
          \boxes[gray]{2}{3}   
          \boxes{3}{1}    
          \draw[line width = 2pt] (1- .05,0) -- (3 + .05,0);
        \end{tikzpicture}
\]
\item If \(f_{u+1} +f_{u+2}=h\), then leave \((f_u, f_{u+1})\)
  unchanged, but change \(u\) to \(u+1\) (\emph{focus shift}).
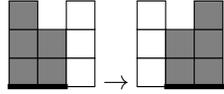
\[
\begin{tikzpicture}[scale = .38]
          \draw (1,0) -- (4,0);
          \boxes[gray]{1}{3}   
          \boxes[gray]{2}{2}   
          \boxes{3}{3}    
          \draw[line width = 2pt] (1- .05,0) -- (3 + .05,0);
        \end{tikzpicture}
        \rightarrow
        \begin{tikzpicture}[scale = .38]
          \draw (1,0) -- (4,0);
          \boxes{1}{3}   
          \boxes[gray]{2}{2}   
          \boxes[gray]{3}{3}     
          \draw[line width = 2pt] (2- .05,0) -- (4 + .05,0);
        \end{tikzpicture}
\]
\end{enumerate}

Let \( \ppm_u^{(m)}(f) \) denote the frequency sequence obtained from
\( (f_u, f_{u+1}) \) after \( m \) particle motions (not counting
focus shifts) have been applied starting from index \(u\). Moreover,
if \( \ppm_u^{(m)}(f) = (\overline{f}_i)_{i \in \Z} \) and the final
focus is on the pair \( (\overline{f}_v, \overline{f}_{v+1}) \), then
we say that the pair \( (f_u, f_{u+1}) \) \emph{moves to}
\( (\overline{f}_v, \overline{f}_{v+1}) \).

Note that by construction,
\( (\overline{f}_i)_{i \in \Z} = \ppm_u^{(m)}(f) \) satisfies
\( \overline{f}_i + \overline{f}_{i+1} \leq h \) for all
\( i \geq u \). Each particle motion increases the weight of the
partition by $1$ (as it removes a part $u$ and adds a part $u+1$),
while focus shifts leave the partition unchanged. Both leave the
length of the partition unchanged.

For example, if \( f = (\dots, 0, f_1=4, f_2=0,2,1,3,1,0,\dots) \),
then \(\ppm_1^{(5)}(f) = (\dots, 0, \overline{f}_1=2,1,3,1,1,3,0,\dots)\), and
the pair \((f_1,f_2)\) moves to \((\overline{f}_5,\overline{f}_6)\). The process is shown in
Figure \ref{fig:1}.
  \begin{figure}[h]
    \centering
    \begin{align*}
      \begin{tikzpicture}[scale = .38]
        \draw (0,0) -- (8,0);
        \boxes[gray]{1}{4}   
        \boxes{3}{2}
        \boxes{4}{1}     
        \boxes{5}{3}   
        \boxes{6}{1}   
        \draw[line width = 2pt] (1- .05,0) -- (3 + .05,0);
      \end{tikzpicture}
      &\Rightarrow
        \begin{tikzpicture}[scale = .38]
          \draw (0,0) -- (8,0);
          \boxes[gray]{1}{3}   
          \boxes[gray]{2}{1}   
          \boxes{3}{2}
          \boxes{4}{1}     
          \boxes{5}{3}   
          \boxes{6}{1}   
          \draw[line width = 2pt] (1- .05,0) -- (3 + .05,0);
        \end{tikzpicture}
        \Rightarrow
        \begin{tikzpicture}[scale = .38]
          \draw (0,0) -- (8,0);
          \boxes[gray]{1}{2}   
          \boxes[gray]{2}{2}   
          \boxes{3}{2}
          \boxes{4}{1}     
          \boxes{5}{3}   
          \boxes{6}{1}   
          \draw[line width = 2pt] (1- .05,0) -- (3 + .05,0);
        \end{tikzpicture}
        \rightarrow
        \begin{tikzpicture}[scale = .38]
          \draw (0,0) -- (8,0);
          \boxes{1}{2}   
          \boxes[gray]{2}{2}   
          \boxes[gray]{3}{2}
          \boxes{4}{1}   
          \boxes{5}{3}   
          \boxes{6}{1}   
          \draw[line width = 2pt] (2- .05,0) -- (4 + .05,0);
        \end{tikzpicture}
        \Rightarrow
        \begin{tikzpicture}[scale = .38]
          \draw (0,0) -- (8,0);
          \boxes{1}{2}   
          \boxes[gray]{2}{1}   
          \boxes[gray]{3}{3}
          \boxes{4}{1}   
          \boxes{5}{3}   
          \boxes{6}{1}   
          \draw[line width = 2pt] (2- .05,0) -- (4 + .05,0);
        \end{tikzpicture}\\
      \rightarrow
      \begin{tikzpicture}[scale = .38]
          \draw (0,0) -- (8,0);
          \boxes{1}{2}   
          \boxes{2}{1}   
          \boxes[gray]{3}{3}
          \boxes[gray]{4}{1}   
          \boxes{5}{3}   
          \boxes{6}{1}   
          \draw[line width = 2pt] (3- .05,0) -- (5 + .05,0);
      \end{tikzpicture}
      &\rightarrow
        \begin{tikzpicture}[scale = .38]
          \draw (0,0) -- (8,0);
          \boxes{1}{2}   
          \boxes{2}{1}   
          \boxes{3}{3}
          \boxes[gray]{4}{1}   
          \boxes[gray]{5}{3}   
          \boxes{6}{1}   
          \draw[line width = 2pt] (4- .05,0) -- (6 + .05,0);
      \end{tikzpicture}
        \rightarrow
        \begin{tikzpicture}[scale = .38]
          \draw (0,0) -- (8,0);
          \boxes{1}{2}   
          \boxes{2}{1}   
          \boxes{3}{3}
          \boxes{4}{1}   
          \boxes[gray]{5}{3}   
          \boxes[gray]{6}{1}   
          \draw[line width = 2pt] (5- .05,0) -- (7 + .05,0);
        \end{tikzpicture}
      \Rightarrow
     \begin{tikzpicture}[scale = .38]
          \draw (0,0) -- (8,0);
          \boxes{1}{2}   
          \boxes{2}{1}   
          \boxes{3}{3}
          \boxes{4}{1}   
          \boxes[gray]{5}{2}   
          \boxes[gray]{6}{2}   
          \draw[line width = 2pt] (5- .05,0) -- (7 + .05,0);
        \end{tikzpicture}
      \Rightarrow
        \begin{tikzpicture}[scale = .38]
          \draw (0,0) -- (8,0);
          \boxes{1}{2}   
          \boxes{2}{1}   
          \boxes{3}{3}
          \boxes{4}{1}   
          \boxes[gray]{5}{1}   
          \boxes[gray]{6}{3}
         \vphantom{\boxes[gray]{7}{4}}
          \draw[line width = 2pt] (5- .05,0) -- (7 + .05,0);
        \end{tikzpicture}       
    \end{align*}
    \caption{Illustration of applying \( 5 \) particle motions
      starting from index $1$ in the frequency sequence
      \(f = (\dots, 0, f_1=4, f_2=0,2,1,3,1,0,\dots) \). The symbol
      \( \Rightarrow \) indicates a particle motion, and
      \( \rightarrow \) indicates a focus shift.}
    \label{fig:1}
  \end{figure}
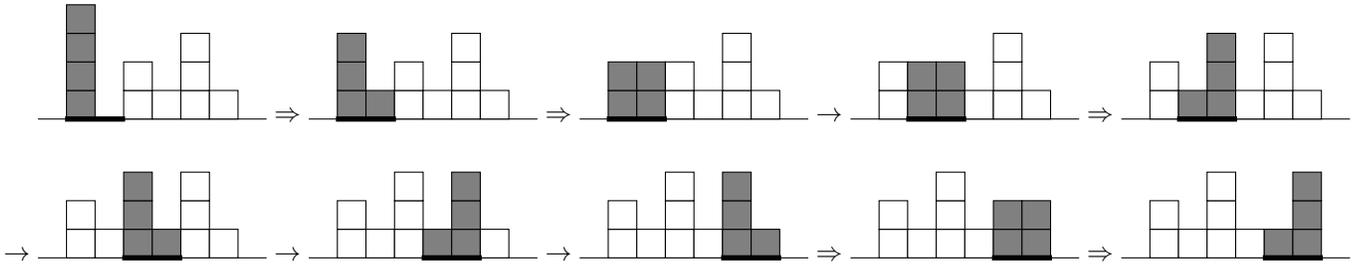
  
  To obtain the left-hand side of \eqref{eq:AG}, Warnaar's idea is to
  start with a minimal partition satisfying \eqref{eq:comb_AG} whose
  weight is \(s_1^2+\dots+s_{k}^2+s_{k-r+1}+\dots+s_{k}\), generating
  the numerator of \eqref{eq:AG}. Then, noting that
  \(\frac{1}{(q)_{s_1-s_2}\dots(q)_{s_{k-1}-s_{k}}(q)_{s_{k}}}\)
  generates \(k\)-tuples of partitions
  \(\bla=(\lambda^{(1)},\dots,\lambda^{(k)})\) where for all \( i \),
  \( \lambda^{(i)} \) has \( s_i - s_{i+1} \) \textit{non-negative}
  parts (after setting \( s_{k+1} := 0 \)), the parts of \(\bla\) are
  inserted into the minimal partition via particle motions. More
  precisely, Warnaar's result is equivalent to the following.

\begin{thm}[Warnaar~\cite{WarnaarParticle}, reformulated]\label{thm:Warnaar}
  For integers \( k \geq 1 \) and \( 0 \leq r \leq k \), let
  \[
    \mathcal{O}_{r,k} := \{(\bla, \tilde{\fs}_{r,k}(\bla)) : \text{\(\bla=(\lambda^{(1)},\dots,\lambda^{(k)})\) is a \(k\)-tuple of partitions into non-negative parts}\},
  \]
where
\begin{align*}
  \tilde{\fs}_{r,k}(\bla):= &(\dots,0,0,\underbrace{f_1=k-r,f_2=r,\dots,k-r,r}_{s_k ~\text{pairs}},\dots, 
  \\&\underbrace{\min(k-r,i),\max(0,i-k+r),\dots,\min(k-r,i),\max(0,i-k+r)}_{s_i - s_{i+1} ~\text{pairs}}, \dots,\underbrace{1,0,\dots,1,0}_{s_1 - s_2 ~\text{pairs}},0,0, \dots ),
\end{align*}
and for all \( i \), \( \lambda^{(i)} \) has length
\( s_i - s_{i+1} \) (after setting \(s_{k+1}=0\)).

Let 
\[
  \mathcal{B}_{r,k} := \{(f_i)_{i \in \Z} : \text{\(f_i + f_{i+1} \leq k\) for all \(i\), \(f_1 \leq k-r\), and \(f_i=0\) for all \(i \leq 0\)} \}.
\]

Define
\[
\tilde{\Lambda}( \bla, \tilde{\fs}_{r,k}(\bla)) =\left(\ppm_{1}^{(\lambda_0)} \circ \ppm_{3}^{(\lambda_1)} \circ \cdots \circ \ppm_{2s_1 -3}^{(\lambda_{s_1 -2})} \circ \ppm_{2s_1 -1}^{(\lambda_{s_1 -1})}\right) \left(\tilde{\fs}_{r,k}(\bla)\right),
\]
where we renamed
\[
  (\lambda_{s_1-1},\dots,\lambda_1,\lambda_0) :=
  (\lambda_1^{(1)},\dots,\lambda_{s_1 - s_2}^{(1)},\dots,\lambda_1^{(k-1)},\dots,\lambda_{s_{k-1} - s_k}^{(k-1)},\lambda_{1}^{(k)},\dots,\lambda_{s_k}^{(k)}).
\] 
Then \(\tilde{\Lambda}\) is a weight-preserving bijection between
\(\mathcal{O}_{r,k}\) and \(\mathcal{B}_{r,k}\).
\end{thm}
The partitions of \(\mathcal{B}_{r,k}\) being exactly those satisfying
the frequency conditions \eqref{eq:comb_AG}, and those of
\(\mathcal{O}_{r,k}\) being generated by the left-hand side of
\eqref{eq:AG}, this proves bijectively that the generating function
for partitions satisfying \eqref{eq:comb_AG} is indeed the left-hand
side of \eqref{eq:AG}.

Here we define a map \(\Lambda\) which generalises Warnaar's map
\(\tilde{\Lambda}\), and which we will use in different contexts.

\begin{defn}
  Let \(\mathcal{P}\) be the set of pairs
  \((\bla, (f_i)_{i \in \Z})\), where
  \(\bla=(\lambda^{(1)},\dots,\lambda^{(k)})\) is a \(k\)-tuple of
  partitions (for some \(k \geq 0\)) into non-negative parts of total
  length \(\ell\), and \((f_i)_{i \in \Z}\) is a generalised frequency
  sequence where there exists some index \(u\) such that
\begin{itemize}
\item for all \(i \notin \{u, u+1, \dots, u+2\ell-1\}\), \(f_i=0\),
\item for all \(i \in \{0, \dots, \ell-1\}\), \(f_{u+2i} + f_{u+2i+1} \geq f_{u+2i+2} + f_{u+2i+3}\).
\end{itemize}
Let \( \ell_i \) be the length of \( \lambda^{(i)} \) and, as before,
rename
\[
  (\lambda_{\ell-1},\dots,\lambda_1,\lambda_0) :=
  (\lambda_1^{(1)},\dots,\lambda_{\ell_1}^{(1)},\dots,\lambda_{1}^{(k)},\dots,\lambda_{\ell_k}^{(k)}).
\]  

Then the map \(\Lambda\) is defined on \(\mathcal{P}\) as follows:
\(\Lambda(\bla, (f_i)_{i \in \Z})\) is the generalised frequency
sequence obtained from \((f_i)_{i \in \Z}\) by applying, for all
\(i \in \{0, \dots, \ell-1\}\), \(\lambda_i\) particle motions at
index \(u+2i\), starting from the right (i.e. starting with
\(i=\ell-1\) and ending with \(i=0\)). In other words,
\[
  \Lambda(\bla, (f_i)_{i \in \Z}) = \left(\ppm_{u}^{(\lambda_0)} \circ \ppm_{u+2}^{(\lambda_1)} \circ \cdots \circ \ppm_{u+2\ell-2}^{(\lambda_{\ell-1})}\right)((f_i)_{i \in \Z}).
\]
\end{defn}

In \cite{DJK2024}, the first author, Jouhet and Konan considered the
particular case where in the definition above, \((f_i)_{i \in \Z}\) is
the \emph{frame sequence} associated with \(\bla\), namely:
\[
\fs(\bla) = (\dots, 0, \underbrace{f_0=k,f_1=0,\dots,k,0}_{s_k ~\text{pairs}},\dots,\underbrace{i,0,\dots,i,0}_{s_i - s_{i+1} ~\text{pairs}},\dots,\underbrace{1,0,\dots,1,0}_{s_1 - s_2 ~\text{pairs}},0, \dots ),
\]
which has weight \(s_1^2+\dots+s_{k}^2-s_{1}-\dots-s_{k}\), where for
all \( i \), \( \lambda^{(i)} \) has \( s_i - s_{i+1} \) parts.

Then they showed the following.
\begin{thm}[\cite{DJK2024}]
\label{th:bijDJK24}
The map \(\Lambda\) induces a bijection between the sets 
\[
  \mathcal{P}_k = \left\{ (\bla, \fs(\bla)) :  \bla \text{ is a  \(k\)-tuple of partitions} \right\},
\]
and
\[
  \mathcal{A}_k = \left\{ (f_i)_{i \in \Z} :  f_i + f_{i+1} \leq k \text{ for all } i \geq 0 \text{ and \(f_i =0\) for all \(i <0\)} \right\}.
\]
\end{thm}

Considering well-chosen subsets of \( \mathcal{A}_k \) and
\( \mathcal{P}_k \), they deduced bijections for the sum-side of
several identities such as Bressoud's identity \cite[(3.4)]{Br80}
  \begin{equation}\label{eq:Br}
    \sum_{s_1\geq\dots\geq s_{k}\geq0}\frac{q^{s_1^2+\dots+s_{k}^2+s_{k-r+1}+\dots+s_{k}}}{(q)_{s_1-s_2}\dots(q)_{s_{k-1}-s_{k}}(q^2;q^2)_{s_{k}}}=\frac{(q^{k+1-r},q^{k+1+r},q^{2k+2};q^{2k+2})_\infty}{(q)_\infty},
  \end{equation}
  which is an even moduli counterpart to the Andrews--Gordon
  identities, or another identity of Bressoud~\cite[(3.3)]{Br80}
\begin{equation}\label{eq:Br3.3}
\sum_{s_1\geq\dots\geq s_{k}\geq0}\frac{q^{s_1^2+\dots+s_{k}^2-s_{1}-\dots-s_{j}}}{(q)_{s_1-s_2}\dots(q)_{s_{k-1}-s_{k}}(q)_{s_{k}}}=\sum_{s=0}^{j}\frac{(q^{k+2-j+2s},q^{k+1+j-2s},q^{2k+3};q^{2k+3})_\infty}{(q)_\infty}.
\end{equation}
Using this bijection and identities \eqref{eq:AG} and \eqref{eq:Br} as
ingredients, this allowed them to prove several new and known
\(q\)-series identities. The \(0\) frequencies in the frame sequences
were key in obtaining the minus signs in the powers in the numerator of the
left-hand side of \eqref{eq:Br3.3}, for example.

Then, in \cite{DJJ25} the two authors and Jouhet reformulated the
bijection \(\Lambda\) in its form presented here, and used it on
subsets of \( \mathcal{A}_k \) and \( \mathcal{P}_k \) with more
intricate conditions, to prove other \(q\)-series identities, known
and new, among which Theorem \ref{thm:Sta3.2}.

In this paper, our goal is to use particle motion to prove identities
involving partitions with parity restrictions. To prove some of these
identities, we will need some partitions where the parts \(-1\) are
allowed. Therefore, we extend Theorem \ref{th:bijDJK24} by allowing
any offset of the frame sequence to the left or to the right.

\begin{defn}
\label{def:u_0-stuff}
Let \(u \in \Z\) and let \(\bla=(\lambda^{(1)},\dots,\lambda^{(k)})\)
be a \(k\)-tuple of partitions (for some \(k \geq 0\)) into
non-negative parts where for all \( i \), \( \lambda^{(i)} \) has
\( s_i - s_{i+1} \) parts. The \(u\)-\emph{frame sequence} associated
with \(\bla\) is the generalised frequency sequence
\[
\fs_{u}(\bla) = (\cdots, 0, \underbrace{ f_{u}=k,f_{u+1}=0,\dots,k,0}_{s_k ~\text{pairs}},\dots,\underbrace{i,0,\dots,i,0}_{s_i - s_{i+1} ~\text{pairs}},\dots,\underbrace{1,0,\dots,1,0}_{s_1 - s_2 ~\text{pairs}},0, \dots ),
\]
which has weight \(s_1^2+\dots+s_{k}^2 +(u-1)( s_{1}+\dots+s_{k})\).
\end{defn}

\begin{thm}
\label{th:u0-bij}
Let \(u \in \Z\), \(k \in \Z_{>0}\). Define
\[
  \mathcal{P}_{k,u} = \left\{ (\bla, \fs_{u}(\bla)) :  \bla \text{ is a  \(k\)-tuple of partitions} \right\},
\]
and
\[
  \mathcal{A}_{k,u} = \left\{ (f_i)_{i \in \Z} :  f_i + f_{i+1} \leq k \text{ for all } i \in \Z, \text{ and } f_i=0 \text{ for all } i<u \right\}.
\]
Then the map \(\Lambda\) induces a bijection between the sets
\(\mathcal{P}_{k,u}\) and \(\mathcal{A}_{k,u}\).
\end{thm}
\begin{proof}
  The proof of Theorem \ref{th:bijDJK24} in \cite{DJK2024} (and
  reformulated in \cite{DJJ25}) actually never used the fact that the
  frame sequence starts at \(0\), only its shape. So the exact same
  proof as in these papers also proves this stronger result.
\end{proof}

Clearly, Theorem \ref{th:u0-bij} generalises Theorem \ref{th:bijDJK24}
(which is the case \(u=0\)). But Theorem \ref{thm:Warnaar} also
follows from it. Indeed, for any \(k\)-tuple of partitions
\(\bla=(\lambda^{(1)},\dots,\lambda^{(k)})\),
\[
  \tilde{\Lambda}(\bla, \tilde{\fs}_{r,k}(\bla)) = \Lambda(\bla', \fs_{1}(\bla')),
\]
where \(\bla'\) is obtained from \(\bla\) by adding \(\max(0,m-k+r)\)
to each part of \(\lambda^{(m)}\), for all \(m \in \{1, \dots, k\}\).
The result follows by seeing that \(\Lambda\) induces a bijection
between
\[
  \left\{ (\bla, \fs_{1}(\bla)) :  \bla=(\lambda^{(1)},\dots,\lambda^{(k)}) \text{and the parts of \(\lambda^{(m)}\) are at least \(\max(0,m-k+r)\) for all \(m\)} \right\} \subset \mathcal{P}_{k,1},
\]
and
\[
  \left\{ (f_i)_{i \in \Z} :  f_i + f_{i+1} \leq k \text{ for all } i \in \Z, \  f_i=0 \text{ for all } i<1, \text{ and } f_1 \leq k-r \right\} \subset \mathcal{A}_{k,1}.
\]

To prove some properties about \(\Lambda\), it is convenient to split it as a composition
of intermediate steps of particle motion.
\begin{defn}\label{def:theta}
  For any \(k\)-tuple of partitions \(\bla\) of total length $s_1$, define
  \begin{equation}\label{eq:theta}
    \fs_u(\bla) =: \theta^{(s_1)} , \theta^{(s_1-1)} , \dots  , \theta^{(1)}, \theta^{(0)} = \Lambda(\bla,\fs_u(\bla)),
  \end{equation}
  where each \( \theta^{(i)} \) is obtained from \( \theta^{(i+1)} \)
  by
  \begin{equation}\label{eq:recursion_theta}
    \theta^{(i)} = \ppm_{u+2i}^{(\lambda_{i})}\left(\theta^{(i+1)}\right) \quad\mbox{for \( i \in \{ s_1 - 1 , \dots,1,0 \} \)}.
  \end{equation}
\end{defn}

Let us start by showing the following simple property.

\begin{prop}
\label{prop:shift}
For any \( k \)-tuple of partitions
\( \bla = (\lambda^{(1)}, \dots, \lambda^{(k)}) \) of total length $s_1$, let \( \bla^+ \)
be the \( k \)-tuple of partitions obtained from \( \bla \) by adding
\( m \) to every part of \( \lambda^{(m)} \) for all
\(m \in \{1, \dots, k\}\). Note that for all \(u \in \Z\),
\(\fs_{u}(\bla) = \fs_{u}(\bla^+)\).
Then for all \(u \in \Z\),
\[ \Lambda(\bla^+, \fs_{u-1}(\bla^+)) =\Lambda(\bla, \fs_{u}(\bla)).\]
\end{prop}
\begin{proof}
Recall that 
\begin{align*}
\Lambda(\bla^+, \fs_{u-1}(\bla^+)) &= \left(\ppm_{u-1}^{(\lambda^+_0)} \circ \ppm_{u+1}^{(\lambda^+_1)} \circ \cdots \circ \ppm_{u+2\ell-3}^{(\lambda^+_{\ell-1})}\right)(\fs_{u-1}(\bla^+)),\\
\Lambda(\bla, \fs_{u}(\bla)) &=\left(\ppm_{u}^{(\lambda_0)} \circ \ppm_{u+2}^{(\lambda_1)} \circ \cdots \circ \ppm_{u+2\ell-2}^{(\lambda_{\ell-1})}\right)(\fs_{u}(\bla)).
\end{align*}
Mimicking \Cref{def:theta}, let us define $\theta^{(s_1),+} =\fs_{u-1}(\bla^+)$, and
\begin{equation*}
\theta^{(i),+} = \ppm_{u+2i-1}^{(\lambda^+_{i})}\left(\theta^{(i+1),+}\right) \quad\mbox{for \( i \in \{ s_1 - 1 , \dots,1,0 \} \)}.
\end{equation*}

We show by downward induction that for all $0 \leq i \leq s_1$,
$\theta^{(i),+}_j$ and $\theta^{(i)}_j$ coincide for  all $j \geq u+2i-1$, and that $\theta^{(i),+}_{u+2i-2}=\theta^{(i),+}_{u+2i-1}=0$.

This is clearly true for $i=s_1$, as all frequencies of $j \geq u+2s_1-1$ are zero in both partitions. Moreover, by definition, the frequency of $u+2i-2$ is $0$ in $\fs_{u-1}(\bla^+)$.

Now assume it is true for $i+1$ and prove it for $i$.
Let $m$ be such that \( \lambda_i\) is a part of \( \lambda^{(m)} \) in
  \( \bla \).
  Then the process $\ppm_{u+2i-1}^{(\lambda_{i}+m)}\left(\theta^{(i+1),+}\right)$ starts
  with the pair \( (m,0) \) at positions \( (u+2i-1 , u+2i) \).
  Through the first \( m \) particle motions, the pair \( (m,0) \) at
  positions \( (u+2i-1, u+2i) \) moves to the pair \( (m,0) \) at
  positions \( (u+2i, u+2i+1) \) (indeed, by the induction hypothesis, before applying the particle motion, the frequencies of $u+2i$ and $u+2i+1$ were $0$). Then the remaining \( \lambda_i\)
  particle motions are applied from this position. This is the
  same as applying directly $\ppm_{u+2i}^{(\lambda_{i})}\left(\theta^{(i+1)}\right)$, as by the induction hypothesis, all the frequencies $j \geq u+2i+1$ coincided in $\theta^{(i+1),+}$ and $\theta^{(i+1)}$. Hence, the frequencies of all $j \geq u+2i-1$ coincide in $\theta^{(i),+}$ and $\theta^{(i)}$.
Finally, because  the first $m$ particle motions in $\ppm_{u+2i-1}^{(\lambda_{i}+m)}\left(\theta^{(i+1),+}\right)$  moved the pair \( (m,0) \) at positions \( (u+2i-1 , u+2i) \) to the pair \( (m,0) \) at positions \( (u+2i, u+2i+1) \), the frequency $\theta^{(i),+}_{u+2i-1}$ becomes zero. Moreover, the frequency $\theta^{(i),+}_{u+2i-2}$ is zero by definition of $\fs_{u-1}(\bla^+)$.
Hence the property is proved by induction, and
$\theta^{(0),+}= \Lambda(\bla^+, \fs_{u-1}(\bla^+)) \) is the same as
  \(\theta^{(0)}= \Lambda(\bla, \fs_{u}(\bla)) \).  
\end{proof}

We now rewrite some of the results of \cite{DJJ25} but extended to the
context of generalised frequency sequences and \(u\)-frame sequences.
None of the proofs used the fact that the frame sequence started at
index \(0\), so they work in the exact same way and we do not repeat
them here. However we write the statement of the results in the more
general context.

The first proposition is a restatement of
\cite[Proposition~5.4]{DJJ25} in this context, and with \( v \)
replaced by \( v+2 \).
\begin{prop}[{\cite[Proposition~5.4]{DJJ25}}]\label{pro:explicit_pm}
  Let \( f = (f_i )_{i \in \Z} \) be a generalised frequency sequence.
  Suppose that \( u \) is an integer such that there exists
  \( h \geq 1 \) with \( (f_u, f_{u+1}) = (h,0) \) and
  \( f_i + f_{i+1} \leq h \) for all \( i \geq u \). For a
  non-negative integer \(m\), let
  \((\overline{f}_i)_{i \in \Z} = \ppm_u^{(m)}(f)\). Define
  \begin{equation}\label{eq:u_bar}
    v := \min \left\{ t \geq u : \sum_{i = u+2}^{t+2} \left( h - (f_{i -1} + f_{i}) \right) \geq m \right\}.
  \end{equation}
  Then \( (f_u, f_{u+1}) \) moves to
  \( (\overline{f}_v, \overline{f}_{v +1}) \). Moreover
  \((\overline{f}_i)_{i \in \Z}\) is given explicitly by:
  \begin{align}
    \label{eq:f_bar}
    \overline{f}_i =
    \begin{cases}
      f_i & \mbox{if \( i < u \),}\\
      f_{i+2} & \mbox{if \( u \leq i < v \),} \\
      f_{v+2} + \sum_{j = u+2}^{v+2} (h - (f_{j-1} + f_j)) - m & \mbox{if \( i = v \),} \\
      f_{v+1} + m - \sum_{j = u+2}^{v +1} (h - (f_{j-1} + f_j)) & \mbox{if \( i = v +1 \),} \\
      f_i & \mbox{if \( i \geq v+2 \).}
    \end{cases}
  \end{align}
\end{prop}

We define three sets which play a key role in all the proofs of this
paper.
\begin{defn}\label{def:XYZ}
  Let \( j,r \), and \( k \) be non-negative integers such that
  \( j+r \leq k \). Define
  \begin{itemize}
  \item \( \mathcal{X}_{j,r,k,u} \) to be the set of all pairs
    \( (\bla, \fs_u(\bla))\), where
    \( \bla = (\lambda^{(1)},\dots,\lambda^{(k)}) \) is a
    \( k \)-tuple of partitions and \( \fs_u(\bla) \) is the \(u\)-frame
    sequence corresponding to \( \bla \), subject to the condition
    that each part of \( \lambda^{(m)} \) is at least
    \( m - j + \max\{m-(k-r), 0\} \) for each \( m = 1,\dots,k \),
  \item \( \mathcal{Y}_{j,r,k,u} \) to be the set of all generalised
    frequency sequences \( (f_i)_{i \in \Z} \) such that
    \( f_i + f_{i+1} \leq k \) for all \( i \geq u \), \( f_i =0\) for
    all \( i < u \), and
  \[
    f_u \in \{\ell+\max\{\ell-(j-r),0 \} : 0 \leq \ell \leq j\},
  \]
\item \( \mathcal{Z}_{j,r,k,u} \) to be the set of all generalised
  frequency sequences \( (f_i)_{i \in \Z} \) such that
  \( f_i + f_{i+1} \leq k \) for all \( i \geq u \), \( f_i =0\) for
  all \( i < u \), and
  \[
    f_u \leq j - \max\{f_u + f_{u+1} - (k-r),0\}.
  \]
  \end{itemize}
\end{defn}

The map \(\Lambda\) restricts to a bijection between two of these
sets.
\begin{prop}[\cite{DJJ25}, Proposition~6.8]\label{pro:bij_XZ}
  The map \( \Lambda \) is a weight-preserving bijection from
  \( \mathcal{X}_{j,r,k,u} \) to \( \mathcal{Z}_{j,r,k,u} \).
\end{prop}

Taking the notation of \Cref{def:theta}, we prove a few results concerning parity conditions on
\( \{ \theta^{(s_1)}, \dots, \theta^{(0)}\} \), in order to show that \(\Lambda\) preserves parity conditions.

\begin{lem}\label{lem:pm_parity}
  Let \(u \in 2\Z\). Let \(\bla\) be a \(k\)-tuple of partitions. For
  all \( i \in \{s,\dots,0\} \), every odd part of
  \( \theta^{(i+1)} \) appears an even number of times and
  \( \lambda_i \) is even if and only if every odd part of
  \( \theta^{(i)} \) appears an even number of times.
\end{lem}
\begin{proof}
  First show that if every odd part of \( \theta^{(i+1)} \) appears an
  even number of times and \( \lambda_i \) is even, then every odd
  part of \( \theta^{(i)} \) appears an even number of times. Let
  \( (\theta^{(i+1)}_{u+2i}, \theta^{(i+1)}_{u+2i+1}) = (h,0) \) for
  some \( h >0 \). Suppose that, applying
  \( \ppm_{u+2i}^{(\lambda_{i})} \), this pair
  \( (\theta^{(i+1)}_{u+2i}, \theta^{(i+1)}_{u+2i+1}) \) moves to
  \( (\theta^{(i)}_v, \theta^{(i)}_{v+1}) \). By
  \Cref{pro:explicit_pm}, all entries \( \theta^{(i)}_j \) for
  \(j \not\in \{v, v+1\}\) either remain at the same position or are
  shifted by two from their original position in \( \theta^{(i+1)} \).
  Hence the parity of all \( \theta^{(i)}_j \) for
  \(j \not\in \{v, v+1\}\) is preserved, and it suffices to show the
  result for \( \theta^{(i)}_v \) and \( \theta^{(i)}_{v+1} \).

  From \eqref{eq:recursion_theta}, we have
  \( |\theta^{(i+1)}| + \lambda_i = |\theta^{(i)}| \). Using
  \eqref{eq:f_bar}, we obtain
  \[
    \sum_{ t\in \Z } (t\,\theta^{(i+1)}_t) + \lambda_i = \sum_{ t < u+2i } (t\,\theta^{(i+1)}_t) + \sum_{u+2i \leq t < v } (t\,\theta^{(i+1)}_{t+2}) + v\cdot\theta^{(i)}_v + (v+1)\cdot\theta^{(i)}_{v+1} + \sum_{ v+2 \leq t } (t\,\theta^{(i+1)}_t).
  \]
  Taking modulo \( 2 \), we obtain
  \begin{equation}\label{eq:parity_eq}
    (u+2i)\theta^{(i+1)}_{u+2i}+ (u+2i+1)\theta^{(i+1)}_{u+2i+1} + \lambda_i \equiv \lambda_i \equiv  v\cdot\theta^{(i)}_v + (v+1)\cdot\theta^{(i)}_{v+1} \pmod{2}.
  \end{equation}
  Since \( \lambda_i \) is even by assumption, and exactly one of
  \( v \) and \( v+1 \) is odd, then this odd part must appear an even
  number of times, which completes the proof.
  
  \medskip Now we show the reverse implication. We recover
  \( \theta^{(i+1)} \) from \( \theta^{(i)} \) as follows. Among all
  indices \( j \geq 2i \), let \( v \) be the smallest index such that
  \( \theta^{(i)}_j + \theta^{(i)}_{j+1} \) has the maximum value. We
  then apply reverse particle motion (see
  \cite[Definition~5.8]{DJJ25}) to \( \theta^{(i)} \) at the pair
  \( (\theta^{(i)}_v, \theta^{(i)}_{v+1}) \) until it moves at
  positions \( 2i \) and \( 2i+1 \) and the second entry becomes
  \( 0 \). The number of such inverse moves is exactly
  \( \lambda_i \). As in the forward direction, all other entries
  except the pair
  \( (\theta^{(i+1)}_{u+2i}, \theta^{(i+1)}_{u+2i+1}) \) either remain
  at the same position or are shifted by two positions, so the parity
  of all other odd parts is preserved. Since
  \( \theta^{(i+1)}_{u+2i+1} = 0 \), it follows that all odd parts of
  \( \theta^{(i+1)} \) appear an even number of times.

  Finally, for \( \lambda_i \), the claim follows directly from
  \eqref{eq:parity_eq}. By the assumption on \( \theta^{(i)} \),
  \[
    v\cdot\theta^{(i)}_v + (v+1)\cdot\theta^{(i)}_{v+1} \equiv 0 \pmod{2}.
  \]
  Hence by~\eqref{eq:parity_eq}, \( \lambda_i \) is even, completing
  the proof.
\end{proof}

By a simple shifting argument, a similar lemma holds for \(u\) odd.

\begin{lem}\label{lem:pm_parity_odd}
  Let \(u \in 2\Z+1\). Let \(\bla\) be a \(k\)-tuple of partitions.
  For all \( i \in \{s,\dots,0\} \), every even part of
  \( \theta^{(i+1)} \) appears an even number of times and
  \( \lambda_i \) is even if and only if every even part of
  \( \theta^{(i)} \) appears an even number of times.
\end{lem}
\begin{proof}
  Let \( \tilde{\theta}^{(i+1)} \) be the frequency sequence obtained
  by shifting \( \theta^{(i+1)} \) one step to the right, i.e.
  replacing \( \theta^{(i+1)}_j \) by \( \theta^{(i+1)}_{j-1} \) for
  all \(j \in \Z\). Then by the definition of particle motion,
  \(\theta^{(i)}\) is exactly the generalised frequency sequence
  obtained by shifting
  \(\tilde{\theta}^{(i)} =
  \ppm_{u+2i+1}^{(\lambda_{i})}\left(\tilde{\theta}^{(i+1)}\right)\)
  one step to the left. The result follows by applying Lemma
  \ref{lem:pm_parity} to the \( \tilde{\theta}^{(i)} \)'s, as the odd
  parts of \( \tilde{\theta}^{(i)} \) become the even parts of
  \( \theta^{(i)}\).
\end{proof}

Let \( \shift(\mathcal{X}_{j,r,k,u}) \) be the set of all pairs
\( (\bla, \fs_{u-1}(\bla))\), where
\( \bla = (\lambda^{(1)},\dots,\lambda^{(k)}) \) is a
\( k \)-tuple of partitions such that each part of \( \lambda^{(m)} \) is
at least \( m + \max\{m-j,0\} + \max\{m-(k-r), 0\} \) for each
\( m = 1,\dots,k \). In other words,
\[
  \shift(\mathcal{X}_{j,r,k,u}) = \left\{ (\bla^+, \fs_{u-1}(\bla)) : (\bla, \fs_{u}(\bla)) \in \mathcal{X}_{j,r,k,u} \right\}.
\]

The following lemma shows the shift invariance property.
\begin{prop}\label{lem:shift_invar}
  The map \( \Lambda \) induces a bijection between
  \( \shift(\mathcal{X}_{j,r,k,u}) \) and \( \mathcal{Z}_{j,r,k,u} \).
\end{prop}
\begin{proof}
  By definition of \(\shift(\mathcal{X}_{j,r,k,u})\), we have
  \((\bla, \fs_{u}(\bla)) \in \mathcal{X}_{j,r,k,u}\) if and only if
  \((\bla^+, \fs_{u-1}(\bla)) \in \shift(\mathcal{X}_{j,r,k,u})\). By
  Proposition \ref{prop:shift}, for any
  \((\bla^+, \fs_{u-1}(\bla)) \in \shift(\mathcal{X}_{j,r,k,u})\), we
  have
  \( \Lambda(\bla^+, \fs_{u-1}(\bla)) = \Lambda(\bla,
  \fs_{u}(\bla))\). And by \Cref{pro:bij_XZ}, \(\Lambda\) induces a
  bijection between \(\mathcal{X}_{j,r,k,u}\) and
  \( \mathcal{Z}_{j,r,k,u} \). Hence it also induces a bijection
  between \( \shift(\mathcal{X}_{j,r,k,u}) \) and
  \( \mathcal{Z}_{j,r,k,u} \).
\end{proof}

The next lemma follows directly from Lemma \ref{lem:pm_parity_odd}.

\begin{lem}\label{lem:shifted_pm_parity}
  For \( (\bla, \fs_{u-1}(\bla)) \in \shift(\mathcal{X}_{j,r,k,u}) \), let
  \[
    \fs_{u-1}(\bla) =: \theta^{(s_1)} , \theta^{(s_1-1)} , \dots  , \theta^{(1)}, \theta^{(0)} = \Lambda(\bla,\fs_{u-1}(\bla)),
  \]
  where each \( \theta^{(i)} \) is obtained from \( \theta^{(i+1)} \)
  by
  \( \theta^{(i)} =
  \ppm_{u+2i-1}^{(\lambda_{i})}\left(\theta^{(i+1)}\right) \) for
  \( i = s_1 - 1 , \dots,1,0 \) . Then every even part of
  \( \theta^{(i+1)} \) appears an even number of times and
  \( \lambda_i \) is even if and only if every even part of
  \( \theta^{(i)} \) appears an even number of times.
\end{lem}

We saw in Proposition \ref{pro:bij_XZ} that \(\Lambda\) induces a
weight-preserving bijection between \( \mathcal{X}_{j,r,k,u} \) and
\( \mathcal{Z}_{j,r,k,u} \). However \( \mathcal{Z}_{j,r,k,u} \) is
not immediately a set of frequency sequences whose generating function
is well-known. Hence we do one last simple bijection to relate \( \mathcal{Z}_{j,r,k,0} \) it to \( \mathcal{Y}_{j,r,k,0} \), whose generating function can be nicely expressed  as sums of products thanks to the Andrews--Gordon identities.

\begin{prop}[Generalisation of \cite{DJJ25}, Proposition~6.7]
  \label{prop:bij_YZ}
  There exists a weight-preserving bijection from \( \mathcal{Y}_{j,r,k,0} \) to
  \( \mathcal{Z}_{j,r,k,0} \) via the map \( \phi \) defined by
  \[
    (\dots,0,f_0, f_{1},f_{2}, \dots) \mapsto  (\dots,0 ,f'_0, f_{1},f_{2}, \dots),
  \]
  where if \( j \geq r \), then
  \[
    f_0' =
    \begin{cases}
      f_0 & \mbox{if \( f_0 \leq j-r \)},\\
      j-r+\ell & \mbox{if \( f_0 = j-r + 2 \ell \) for some \( \ell \in \{1,\dots,r\} \),}
    \end{cases}
  \]
  and if \( j < r \), then
  \[
    f_0' = \ell, \quad\text{for \( f_0 = r-j+2 \ell \)
      with \( \ell \in \{0,\dots,j\}\).}
  \]
\end{prop}

\section{Proof of Theorem \ref{thm:main} and Corollary \ref{cor:odd}}
\label{sec:proofmain}
In this section, we prove Theorem \ref{thm:main}.
First, we define a subset  of $\mathcal{X}_{j,r,k,0}$ with additional parity conditions.
\begin{defn}\label{def:Xo}
  Let \( j,r \), and \( k \) be non-negative integers with
  \( j+r \leq k \). Define \( \mathcal{X}^o_{j,r,k} \) to be the set
  of all pairs \( (\bla,\fs_0(\bla)) \), where
  \( \bla = (\lambda^{(1)},\dots,\lambda^{(k)}) \) is a
  \( k \)-tuple of partitions, such that for all $m \in \{1, \dots , k\}$,
  \begin{itemize}
  \item each part of \( \lambda^{(m)} \) is at least
    \( m - j + \max\{m-(k-r), 0\} \) for each \( m = 1,\dots,k \),
  \item \( \lambda^{(m)} \) is a partition into even parts.
  \end{itemize}

\end{defn}

Let us start by computing the generating function for \( \mathcal{X}^o_{j,r,k} \).

\begin{prop}\label{pro:1}
Let \( j,r \), and \( k \) be non-negative integers with
  \( j+r \leq k \). Then
  \[
    \sum_{(\bla,\fs_0(\bla)) \in \mathcal{X}_{j,r,k}^o} q^{|(\bla,\fs_0(\bla))|} 
    = \sum_{s_1 \geq \cdots \geq s_k \geq 0}\frac{q^{s_1^2 + \cdots + s_k^2 - (s_1 + \cdots + s_j) + (s_{j+1} - s_{j+2} + s_{j+3} - \cdots \pm s_{k-r}) + (s_{k-r+1} + \cdots + s_k)}}{(q^2;q^2)_{s_1 - s_2} \cdots (q^2;q^2)_{s_{k-1} - s_k} (q^2;q^2)_{s_k}}.
  \]
\end{prop}
\begin{proof}
  Note that the generating function for partitions of length
  \( \ell \) into even parts of size at least \( b \) is 
  \[
    \begin{cases*}
      \frac{q^{b \ell}}{(q^2;q^2)_\ell}& \mbox{if \( b \) is even},\\
      \frac{q^{(b+1) \ell}}{(q^2;q^2)_\ell}& \mbox{if \( b \) is odd.}
    \end{cases*}
  \]

  For non-negative integers \( s_1,\dots,s_k \) with
  \( s_1 \ge \cdots \ge s_k \geq 0 \), define
  \( X_{j,r}^o(s_1,\dots,s_k) \) to be the set of \( k \)-tuples of
  partitions \( \bla = (\lambda^{(1)},\dots,\lambda^{(k)}) \) such
  that \((\bla,\fs_0(\bla)) \in \mathcal{X}_{j,r,k}^o\) and for each
  \( m \), the length of \( \lambda^{(m)} \) is \( s_m - s_{m+1} \).
  Then \(\fs_0(\bla)\) is the same for all
  \(\bla \in X_{j,r}^o(s_1,\dots,s_k)\); let us denote it by
  \( \fs_0(s_1,\dots,s_k) \).
  
  Recall that
  \begin{equation}\label{eq:size_fs}
    |\fs_0(s_1, \dots, s_k)| = s_1^2 + \cdots + s_k^2 - (s_1 + \cdots + s_k).
  \end{equation}

  We can express \( \mathcal{X}_{j,r,k}^o \) as
  \begin{equation}\label{eq:Xo_union}
    \mathcal{X}_{j,r,k}^o=\bigsqcup_{s_1 \geq \dots \geq s_k \geq 0} X_{j,r}^o(s_1,\dots,s_k) \times \{\fs_0(s_1,\dots,s_k)\}.
  \end{equation}

 To compute the generating function for $X_{j,r}^o(s_1, \dots, s_k)$, we distinguish the different cases in $\max\{m-(k-r), 0\}$. We obtain
$$ \sum_{ \bla \in X_{j,r}^o(s_1, \dots, s_k) } q^{|\bla|}
    = A_1 \cdot A_2 \cdot A_3,$$
with
\begin{align*}
A_1 &= \sum q^{|\lambda^{(1)}| + \cdots + |\lambda^{(j)}|},\\
A_2 &= \sum q^{|\lambda^{(j+1)}| + \cdots + |\lambda^{(k-r)}|},\\
A_3 &= \sum q^{|\lambda^{(k-r+1)}| + \cdots + |\lambda^{(k)}|},
\end{align*}
where the sums run respectively on all tuples of partitions $(\lambda^{(1)}, \dots, \lambda^{(j)})$, $(\lambda^{(j+1)}, \dots , \lambda^{(k-r)})$, and $(\lambda^{(k-r+1)}, \dots, \lambda^{(k)})$ such that $(\lambda^{(1)},\dots,\lambda^{(k)}) \in X_{j,r}^o(s_1,\dots,s_k)$.

  For all $m \in \{1, \dots, j\}$, \( \lambda^{(m)} \) is a partition into $s_m - s_{m+1}$ even parts at least
  \( 0 \). Thus
  \begin{align} \label{eq:Xo_term1}
    A_1 = \prod_{ m = 1 }^{ j } \frac{1}{(q^2;q^2)_{s_m - s_{m+1}}}.
  \end{align}

For the computation of $A_2$ and $A_3$, we need to distinguish two cases.
First assume that \( k-r-j \) is even (the case \( k-r-j \) odd will follow). 

For all $m \in \{j+1, \dots, k-r\}$,
  \( \lambda^{(m)} \) is a partition into $s_m - s_{m+1}$ even parts at least \( m-j \), hence
  \begin{multline}\label{eq:Xo_term2}
    A_2 = \frac{q^{2(s_{j+1} - s_{j+2}) + 2(s_{j+2} - s_{j+3}) + 4(s_{j+3} - s_{j+4}) + 4(s_{j+4} - s_{j+5}) + \cdots + (k-r-j)(s_{k-r-1} - s_{k-r}) + (k-r-j)(s_{k-r} - s_{k-r+1})}}{(q^2;q^2)_{s_{j+1} - s_{j+2}}(q^2;q^2)_{s_{j+2} - s_{j+3}} \cdots (q^2;q^2)_{s_{k-r} - s_{k-r+1}}} \\
    = \frac{q^{2s_{j+1} + 2s_{j+3} + \cdots + 2s_{k-r-1} - (k-r-j)s_{k-r+1}}}{(q^2;q^2)_{s_{j+1} - s_{j+2}}(q^2;q^2)_{s_{j+2} - s_{j+3}} \cdots (q^2;q^2)_{s_{k-r} - s_{k-r+1}}}.
  \end{multline}
  
  For all $m \in \{k-r+1, \dots , k\}$,  \( \lambda^{(m)} \) is a partition into even parts at least
  \( 2m - k + r -j \), hence
  \begin{align} \label{eq:Xo_term3}
    A_3 = \prod_{ m = k-r+1 }^{ k } \frac{q^{(2m-k+r-j)(s_m - s_{m+1})}}{(q^2;q^2)_{s_m - s_{m+1}}} = \frac{q^{(k-r-j+2)s_{k-r+1}} q^{2 s_{k-r+2} + 2s_{k-r+3} + \cdots + 2 s_k}}{(q^2;q^2)_{s_{k-r+1} - s_{k-r+2}} \cdots (q^2;q^2)_{s_{k-1} - s_k} (q^2;q^2)_{s_k}}.
  \end{align}
  By combining \eqref{eq:Xo_term1}, \eqref{eq:Xo_term2}, and
  \eqref{eq:Xo_term3}, we obtain
  \[
    \sum_{ \bla \in X_{j,r}^o(s_1, \dots, s_k) } q^{|\bla|} = \frac{q^{2s_{j+1} + 2s_{j+3} + \cdots + 2s_{k-r-1} + 2s_{k-r+1} + 2s_{k-r+2} + \cdots + 2s_k}}{(q^2;q^2)_{s_1 - s_2} \cdots (q^2;q^2)_{s_{k-1} - s_k} (q^2;q^2)_{s_k}}.
  \]
 Combining this with \eqref{eq:size_fs} and \eqref{eq:Xo_union}, we obtain the desired formula.

 Now we treat the case where \( k-r-j \) is odd. All arguments remain the
  same, except that equations~\eqref{eq:Xo_term2} and
  \eqref{eq:Xo_term3} are replaced by
  \begin{multline*}
    \frac{q^{2(s_{j+1} - s_{j+2}) + 2(s_{j+2} - s_{j+3}) + 4(s_{j+3} - s_{j+4}) + 4(s_{j+4} - s_{j+5}) + \cdots + (k-r-j-1)(s_{k-r-1} - s_{k-r}) + (k-r-j+1)(s_{k-r} - s_{k-r+1})}}{(q^2;q^2)_{s_{j+1} - s_{j+2}}(q^2;q^2)_{s_{j+2} - s_{j+3}} \cdots (q^2;q^2)_{s_{k-r} - s_{k-r+1}}} \\
    = \frac{q^{2s_{j+1} + 2s_{j+3} + \cdots + 2s_{k-r}} q^{- (k-r-j+1)s_{k-r+1}}}{(q^2;q^2)_{s_{j+1} - s_{j+2}}(q^2;q^2)_{s_{j+2} - s_{j+3}} \cdots (q^2;q^2)_{s_{k-r} - s_{k-r+1}}},
  \end{multline*}
  and
  \[
    \prod_{ m = k-r+1 }^{ k } \frac{q^{(2m-k+r-j+1)(s_m - s_{m+1})}}{(q^2;q^2)_{s_m - s_{m+1}}} = \frac{q^{(k-r-j+3)s_{k-r+1}} q^{2 s_{k-r+2} + 2s_{k-r+3} + \cdots + 2 s_k}}{(q^2;q^2)_{s_{k-r+1} - s_{k-r+2}} \cdots (q^2;q^2)_{s_{k-1} - s_k} (q^2;q^2)_{s_k}},
  \]
  respectively.
  Hence, we obtain again
  \[
    \sum_{ \bla \in X_{j,r}^o(s_1, \dots, s_k) } q^{|\bla|} = \frac{q^{2s_{j+1} + 2s_{j+3} + \cdots + 2s_{k-r} + 2s_{k-r+1} + 2s_{k-r+2} + \cdots + 2s_k}}{(q^2;q^2)_{s_1 - s_2} \cdots (q^2;q^2)_{s_{k-1} - s_k} (q^2;q^2)_{s_k}},
  \]
 and the result follows.
  
\end{proof}

Now we define a set \( \mathcal{Y}^o_{j,r,k} \) which  will be in bijection with \( \mathcal{Z}^o_{j,r,k} \) and  \( \mathcal{X}^o_{j,r,k} \), and whose generating function is expressed nicely as a sum of products.

\begin{defn}\label{def:Yo}
  For non-negative integers \( j,r \), and \( k \) with
  \( j+r \leq k \), define \( \mathcal{Y}^o_{j,r,k} \) to be the set
  of all frequency sequences \( (f_i)_{i \in \Z} \) such that
  \begin{itemize}
  \item  \( f_i + f_{i+1} \leq k \) for all \( i \geq 0 \),
  \item \(f_i=0\) for all \(i<0\),
  \item \( f_0 \in \{\ell+\max\{\ell-(j-r),0 \} : 0 \leq \ell \leq j\} \), and
  \item \( f_i \) is even if \( i \) is odd. That is, odd parts appear an even number of times.
  \end{itemize}
\end{defn}

Now we prove that these three sets are indeed in bijection.

\begin{prop}\label{prop:bij_XZ_odd}
  There exists a weight-preserving bijection between
  \( \mathcal{X}_{j,r,k}^o \) and \( \mathcal{Z}^o_{j,r,k} \).
\end{prop}
\begin{proof}
  Since \( \mathcal{X}_{j,r,k}^o \subseteq \mathcal{X}_{j,r,k,0} \) and
  \( \mathcal{Z}_{j,r,k}^o \subseteq \mathcal{Z}_{j,r,k,0} \), we will show
  that the bijection
  \( \Lambda : \mathcal{X}_{j,r,k,0} \to \mathcal{Z}_{j,r,k,0} \) (see
  \Cref{pro:bij_XZ}) restricts to a bijection
  \( \Lambda: \mathcal{X}_{j,r,k}^o \to \mathcal{Z}^o_{j,r,k} \).

Let \( (\bla,\fs_0(\bla)) \in \mathcal{X}_{j,r,k}^o \).   We
  use the notation of \Cref{def:theta}. By the definition of
  \( \mathcal{X}_{j,r,k}^o \), every \( \lambda_i \) is even. Thus, by
  iterating \Cref{lem:pm_parity}, it follows that
  \( \theta^{(s_1 -1)}, \dots, \theta^{(1)}\), and
  \( \theta^{(0)} = \Lambda(\bla,\fs_0(\bla)) \) all satisfy the
  condition that every odd part appears an even number of times.
  Therefore, the image of \( \mathcal{X}_{j,r,k}^o \) under
  \( \Lambda \) satisfies the parity condition of
  \( \mathcal{Z}_{j,r,k}^o \), and hence
  \( \Lambda(\mathcal{X}_{j,r,k}^o) \subseteq \mathcal{Z}_{j,r,k}^o
  \).
  For the reverse inclusion, \Cref{lem:pm_parity} also shows that the inverse map $\Lambda^{-1}$
  sends \( \mathcal{Z}_{j,r,k}^o \) into \( \mathcal{X}_{j,r,k}^o \),
  that is,
  \( \Lambda^{-1}(\mathcal{Z}_{j,r,k}^o) \subseteq
  \mathcal{X}_{j,r,k}^o \), which completes the proof.
\end{proof}

\begin{prop}\label{prop:bij_YZ_odd}
  There exists a weight-preserving bijection between
  \( \mathcal{Y}^o_{j,r,k} \) and \( \mathcal{Z}^o_{j,r,k} \).
\end{prop}
\begin{proof}
  The bijection \( \phi \) in \Cref{prop:bij_YZ} for $u=0$ only modifies the
  frequency of $0$, and hence has no effect on the number of appearances of
  the odd parts. Therefore, \( \phi \) restricts to a bijection
  between \( \mathcal{Y}^o_{j,r,k} \) and \( \mathcal{Z}^o_{j,r,k} \).
\end{proof}

Finally, we use \eqref{eq:W_bar} to express the generating function for \( \mathcal{Y}^o_{j,r,k} \) as a sum of products.

\begin{prop}\label{pro:gf Y}
  For non-negative integers \( j,r \), and \( k \) with
  \( j+r \leq k \), we have
  \begin{equation}\label{eq:gf Y}
    \sum_{f \in \mathcal{Y}^o_{j,r,k}} q^{|f|}  
    = \sum_{s=0}^j  \frac{(-q^2;q^2)_\infty(q^{ 2 \lfloor \frac{k+2-r+j-2s}{2} \rfloor}, q^{2k+4 - 2 \lfloor \frac{k+2-r+j-2s}{2} \rfloor} , q^{2k+4}; q^{2k+4})_\infty}{(q^2;q^2)_\infty}.
  \end{equation}
\end{prop}
\begin{proof}
Let us start with the case \( j \leq r \). Then,
  \( \ell+\max\{\ell-(j-r),0 \} = r-j+ 2 \ell \) for
  \( 0 \leq \ell \leq j \). The set of all possible values for \( f_0 \) is
  \[
   F= \{\ell+\max\{\ell-(j-r),0 \} : 0 \leq \ell \leq j\} = \{ r-j+2s: 0 \leq s \leq j\}.
  \]
 Thus
  $$\mathcal{Y}^o_{j,r,k}= \bigsqcup_{s=0}^j \{f \in \mathcal{Y}^o_{j,r,k} : f_0 = r-j+2s\}.$$
  Moreover, the set $\{f \in \mathcal{Y}^o_{j,r,k} : f_0 = a\}$ is in bijection with $\overline{\mathcal{W}}_{k,k-a}$ (by removing the parts $0$). 
  Hence
\begin{align*}
\sum_{f \in \mathcal{Y}^o_{j,r,k}} q^{|f|}  & = \sum_{s=0}^j \sum_{n \geq 0} \overline{W}_{k,k-r+j-2s}(n)q^n,
\end{align*}
and the result follows using \eqref{eq:W_bar}.

  Now suppose \( j>r \). Then the set
  \[
   F= \{\ell+\max\{\ell-(j-r),0 \} : 0 \leq \ell \leq j\}
    = \{ 0, 1,\dots,j-r,j-r+2,\dots,j+r \} 
  \]
  can be expressed as the disjoint union $F=F_1 \sqcup F_2$, where
\begin{align*}
F_1 &= \{ j+r, j+r-2, j+r-4 , \cdots  \}\\
&= \{r-j+2s : s = \lfloor (j-r)/2 \rfloor,\lfloor (j-r)/2 \rfloor+1,\dots,j\},
\end{align*}  
 and
\begin{align*}
F_2 &= \{ j-r-1, j-r-3, j-r-5 , \cdots  \}\\
&= \{j-r-1-2s: s = 0,1,\dots,\lfloor (j-r)/2 \rfloor -1\}.
\end{align*}   
  By \eqref{eq:W_bar}, the corresponding generating functions are, respectively,
  \begin{multline}
    \label{eq:first}
   \sum_{f \in \mathcal{Y}^o_{j,r,k}: f_0 \in F_1} q^{|f|}=  \sum_{s=\lfloor (j-r)/2 \rfloor}^j \sum_{n \geq 0} \overline{W}_{k,k-(r-j+2s)}(n)q^n  \\
    = \sum_{s=\lfloor (j-r)/2 \rfloor}^j \frac{(-q^2;q^2)_\infty(q^{2 \lfloor \frac{k-(r-j+2s)+2}{2} \rfloor}, q^{2k+4-2 \lfloor \frac{k-(r-j+2s)+2}{2} \rfloor}, q^{2k+4}; q^{2k+4})_\infty}{(q^2;q^2)_\infty},
  \end{multline}
  and
  \begin{multline}\label{eq:second}
    \sum_{f \in \mathcal{Y}^o_{j,r,k}: f_0 \in F_2} q^{|f|}=  \sum_{s=0}^{\lfloor (j-r)/2 \rfloor -1} \sum_{n \geq 0} \overline{W}_{k,k-(j-r-1-2s)}(n)q^n  \\
    = \sum_{s=0}^{\lfloor (j-r)/2 \rfloor -1} \frac{(-q^2;q^2)_\infty(q^{2 \lfloor \frac{k-(j-r-1-2s)+2}{2} \rfloor}, q^{2k+4-2 \lfloor \frac{k-(j-r-1-2s)+2}{2} \rfloor}, q^{2k+4}; q^{2k+4})_\infty}{(q^2;q^2)_\infty}.
  \end{multline}
  We claim that
  \[
    2k+4-2 \left\lfloor \frac{k-(j-r-1-2s)+2}{2} \right\rfloor
    = 2 \left\lfloor \frac{k-(r-j+2s)+2}{2} \right\rfloor.
  \]
  If \( k-j+r \) is odd, then
  \[
    2k+4-2 \left\lfloor \frac{k-(j-r-1-2s)+2}{2} \right\rfloor
    = 2 \left( \frac{k+j-r-2s+1}{2} \right)
    = 2 \left\lfloor \frac{k+j-r-2s+2}{2} \right\rfloor,
  \]
  and hence the claim holds. The even case follows similarly. The
  summands in the right-hand sides of \eqref{eq:second} and
  \eqref{eq:first} are the same. Adding \eqref{eq:first} and \eqref{eq:second} together
  completes the proof.
\end{proof}

The proof of Theorem \ref{thm:main} follows directly from combining Propositions \ref{pro:1}, \ref{prop:bij_XZ_odd}, \ref{prop:bij_YZ_odd} and \ref{pro:gf Y}.

\medskip
Let us conclude this section with the proof of \Cref{cor:odd}.
\begin{proof}[Proof of \Cref{cor:odd}]
  First, substituting \( j=a \) and \( r=0 \) into \Cref{thm:main}
  yields
  \begin{multline}\label{eq:main r=0}
    \sum_{s_1 \ge \cdots \ge s_{k} \geq 0}
    \frac{q^{s_1^2 + \cdots + s_{k}^2 - (s_1 + \cdots + s_a) + (s_{a+1} - s_{a+2} + \cdots +(-1)^{k-a-1} s_{k})}}
    {(q^2;q^2)_{s_1 - s_2} \cdots (q^2;q^2)_{s_{k-1} - s_{k}}(q^2;q^2)_{s_{k}}}\\
    = \sum_{s=0}^a  \frac{(-q^2;q^2)_\infty(q^{ 2 \lfloor \frac{k+a-2s+2}{2} \rfloor}, q^{2k+4 - 2 \lfloor \frac{k+a-2s+2}{2} \rfloor} , q^{2k+4}; q^{2k+4})_\infty}{(q^2;q^2)_\infty}.
  \end{multline}
  The left-hand side of the desired identity is obtained immediately.
  On the other hand, the right-hand side is written in a slightly
  different form. Note that \( k-i+2 \leq k+2 \) for
  \( i = 0, \dots, a \), while \( k+a-2s+2 \geq k+2 \) for
  \( s = 0, \dots, \lfloor a/2 \rfloor \), and \( k+a-2s+2 < k+2 \)
  for \( s = \lfloor a/2 \rfloor +1, \dots, a \). To obtain the
  desired right-hand side from \eqref{eq:main r=0}, it suffices to
  show that
  \begin{multline*}
    \left\{ 2\left\lfloor \frac{k-i+2}{2} \right\rfloor : i = 0, \dots, a\right\} 
    = \left\{ 2\left\lfloor \frac{k+a-2s+2}{2} \right\rfloor: s = \lfloor a/2 \rfloor+1, \dots, a\right\} \\
    \sqcup \left\{ 2k+4- 2\left\lfloor \frac{k+a-2s+2}{2} \right\rfloor : s = 0, \dots, \lfloor a/2 \rfloor \right\}.
  \end{multline*}

  By an argument similar to that in the proof of \Cref{pro:gf Y}, an
  elementary property of the floor function yields
  \[
    2k+4- 2\left\lfloor \frac{k+a-2s+2}{2} \right\rfloor = 2\left\lfloor \frac{k-a+2s+3}{2} \right\rfloor.
  \]
  The proof follows from the fact that the set
  \( \{k-i+2: i=0, \dots, a\} \) can be written as a disjoint union
  \begin{multline*}
    \{k-i+2: i=0, \dots, a\} = \left(\{k-a+2, k-a+4, k-a+6, \dots \} \cap \{1, \dots, k+2\}\right) \\
    \sqcup \left(\{k-a+3, k-a+5, k-a+7, \dots \} \cap \{1, \dots, k+2\}\right),
  \end{multline*}
  where the first and second sets are equal to, respectively,
  \[
    \{k+a-2s+2: s = \lfloor a/2 \rfloor+1, \dots, a\}, \qand
    \{k-a+2s+3: s = 0, \dots, \lfloor a/2 \rfloor \}.
  \]
\end{proof}

\section{Proof of Theorem \ref{thm:even1}}
\label{sec:general-case-2}
We now turn to the proof of  Theorem \ref{thm:even1}, which involves different parity conditions. In this section, the $(-1)$-frame sequences will be useful.

First define a subset of \( \shift(\mathcal{X}_{j,r,k,0}) \) with parity conditions.
\begin{defn}\label{def:Xe}
  Let \( a,b \), and \( k \) be non-negative integers with
  \( 2a + 2b \leq k \). Define \( \mathcal{X}^e_{a,b,k} \) to be the
  set of all pairs \( (\bla, \fs_{-1}(\bla)) \), where
  \( \bla = (\lambda^{(1)},\dots,\lambda^{(k)}) \) is a
  \( k \)-tuple of partitions, subject to the conditions that for all $m \in \{1, \dots, k\}$,
  \begin{itemize}
  \item each part of \( \lambda^{(m)} \) is at least
    \( m + \max\{ m - 2a, 0\} + \max\{m-(k-2b), 0\} \) for each \( m = 1,\dots,k \),
  \item \( \lambda^{(m)} \) is a partition into even parts.
  \end{itemize}
\end{defn}

We first compute the generating function for \( \mathcal{X}^e_{a,b,k} \).

\begin{prop}\label{pro:gf_Xe}
Let \( a,b \), and \( k \) be non-negative integers with
  \( 2a + 2b \leq k \). Then
  \[
    \sum_{(\bla, \fs_{-1}(\bla)) \in \mathcal{X}_{a,b,k}^e} q^{|(\bla, \fs_{-1}(\bla))|} 
    = \sum_{ s_1 \geq \dots \geq s_k \geq 0 } \frac{q^{s_1^2 + \cdots + s_k^2 -2(s_2 + s_4 + \cdots + s_{2a}) +2(s_{k-2b+1} + s_{k-2b+3} + \cdots + s_{k-1})}}{(q^2;q^2)_{s_1 - s_2} \cdots (q^2;q^2)_{s_{k-1} - s_k}(q^2;q^2)_{s_k}}.
  \]
\end{prop}
\begin{proof}
  This proof is analogous to that of \Cref{pro:1}. 

  For non-negative integers \( s_1,\dots,s_k \) with
  \( s_1 \ge \cdots \ge s_k \geq 0 \), define
  \( X_{a,b}^e(s_1,\dots,s_k) \) to be the set of \( k \)-tuples of
  partitions \( \bla = (\lambda^{(1)},\dots,\lambda^{(k)}) \) such
  that $(\bla,\fs_0(\bla)) \in \mathcal{X}_{a,b,k}^e$ and for each
  \( m \), the length of \( \lambda^{(m)} \) is \( s_m - s_{m+1} \).
  Then \(\fs_{-1}(\bla)\) is the same for all
  \(\bla \in X_{a,b}^e(s_1,\dots,s_k)\); let us denote it by
  \( \fs_{-1}(s_1,\dots,s_k) \).

    \begin{equation}\label{eq:fs'_size}
|\fs_{-1}(s_1, \dots, s_k)| = s_1^2 + \cdots + s_k^2 - 2 (s_1 + \cdots + s_k).
  \end{equation}
 
  We can express \( \mathcal{X}_{a,b,k}^e \) as the disjoint union
  \begin{equation}\label{eq:Xe_union}
    \mathcal{X}_{a,b,k}^e = \bigsqcup_{s_1 \geq \dots \geq s_k \geq 0} X_{a,b}^e(s_1,\dots,s_k) \times \{\fs_{-1}(s_1,\dots,s_k)\},
  \end{equation}

 To compute the generating function for $X_{a,b}^e(s_1,\dots,s_k)$, we distinguish the different cases in $\max\{ m - 2a, 0\} + \max\{m-(k-2b), 0\}$. We obtain
$$ \sum_{ \bla \in X_{a,b}^e(s_1,\dots,s_k)} q^{|\bla|}
    = A_1 \cdot A_2 \cdot A_3,$$
with
\begin{align*}
A_1 &= \sum q^{|\lambda^{(1)}| + \cdots + |\lambda^{(2a)}|},\\
A_2 &= \sum q^{|\lambda^{(2a+1)}| + \cdots + |\lambda^{(k-2b)}|},\\
A_3 &= \sum q^{|\lambda^{(k-2b+1)}| + \cdots + |\lambda^{(k)}|},
\end{align*}
where the sums run respectively on all tuples of partitions $(\lambda^{(1)}, \dots, \lambda^{(2a)})$, $(\lambda^{(2a+1)}, \dots , \lambda^{(k-2b)})$, and $(\lambda^{(k-2b+1)}, \dots, \lambda^{(k)})$ such that $(\lambda^{(1)},\dots,\lambda^{(k)}) \in X_{a,b}^e(s_1,\dots,s_k)$.

  For all $m \in \{1, \dots, 2a\}$, \( \lambda^{(m)} \) is a partition into $s_m - s_{m+1}$ even parts at least
  \( m \). Thus
\begin{equation}\label{eq:Xe_term1}
    A_1 = \frac{q^{2s_1 + 2s_3 + \cdots + 2s_{2a-1}} q^{-(2a) s_{2a+1}}}{(q^2;q^2)_{s_1 - s_2} \cdots (q^2;q^2)_{s_{2a} - s_{2a+1}}}.
  \end{equation}

For all $m \in \{2a+1, \dots, k-2b\}$,
  \( \lambda^{(m)} \) is a partition into $s_m - s_{m+1}$ even parts at least \( 2m - 2a \), hence 
  \begin{equation}\label{eq:Xe_term2}
    A_2 = \frac{q^{(2a+2)s_{2a+1}}q^{2s_{2a+2} + 2s_{2a+3} + \cdots + 2s_{k-2b}} q^{-(2k-2a-4b) s_{k-2b+1}}}{(q^2;q^2)_{s_{2a+1} - s_{2a+2}} \cdots (q^2;q^2)_{s_{k-2b} - s_{k-2b+1}}}.
  \end{equation}
  
  For all $m \in \{k-2b+1, \dots , k\}$,  \( \lambda^{(m)} \) is a partition into even parts at least
 \( 3m - k -2a + 2b \), hence
  \begin{equation}\label{eq:Xe_term3}
    A_3 = \frac{q^{(2k-2a-4b+4)s_{k-2b+1}}q^{2s_{k-2b+2} + 4s_{k-2b+3} + 2s_{k-2b+4} + 4s_{k-2b+5} + \cdots + 4s_{k-1} + 2s_k}}{(q^2;q^2)_{s_{k-2b+1} - s_{k-2b+2}} \cdots (q^2;q^2)_{s_k}}.
  \end{equation}

  Thus, by \eqref{eq:Xe_term1}, \eqref{eq:Xe_term2}, and
  \eqref{eq:Xe_term3}, we obtain
  \[
    \sum_{ \bla \in X_{a,b}^e(s_1, \dots, s_k) } q^{|\bla|} = \frac{q^{2(s_1 + s_3 + \cdots + s_{2a-1}) + 2(s_{2a+1} + s_{2a+2} + \cdots + s_{k-2b}) + (4s_{k-2b+1} + 2s_{k-2b+2} + 4s_{k-2b+3} + \cdots + 4s_{k-1} + 2s_k)}}{(q^2;q^2)_{s_1 - s_2} \cdots (q^2;q^2)_{s_{k-1} - s_k} (q^2;q^2)_{s_k}}.
  \]
  
  Combining this with \eqref{eq:fs'_size} and \eqref{eq:Xe_union}
  yields the desired formula.
  
\end{proof}

Following the same structure of proof as for Theorem \ref{thm:main}, we  define a set  \( \mathcal{Y}^e_{a,b,k} \) which will be in bijection with \( \mathcal{X}_{a,b,k}^e \) and \( \mathcal{Z}^e_{a,b,k} \).

\begin{defn}\label{def:Ye}
  For non-negative integers \( a,b \), and \( k \) with
  \( 2a + 2b \leq k \), we define \( \mathcal{Y}^e_{a,b,k} \) to be
  the set of all frequency sequences \( (f_i)_{i \in \Z} \) such that
  \begin{itemize}
  \item \( f_i + f_{i+1} \leq k \) for all \( i \geq 0 \),
  \item \(f_i=0\) for all \(i<0\),
  \item \( f_0 \in \{ 2(\ell+\max\{\ell-(a-b),0 \}) : 0 \leq \ell \leq a\} \), and
  \item \( f_i \) is even if \( i \) is even. That is, even parts appear an even number of times.
  \end{itemize}
\end{defn}

Now we prove that these three sets are indeed in bijection.

\begin{prop}\label{prop:bij_XZ_even}
  There exists a weight-preserving bijection between
  \( \mathcal{X}^e_{a,b,k} \) and \( \mathcal{Z}^e_{a,b,k} \).
\end{prop}
\begin{proof}
  The argument parallels that of \Cref{prop:bij_XZ_odd}. Here we set
  \( j = 2a \) and \( r = 2b \). In that proof, replace $ j $ by
  $ 2a$, $ r$ by $ 2b$, \Cref{pro:bij_XZ} by \Cref{lem:shift_invar} ,
  \( \mathcal{X}^o_{j,r,k} \) by\( \mathcal{X}^e_{a,b,k} \) ,
  \( \mathcal{Z}^o_{j,r,k} \) by \( \mathcal{Z}^e_{a,b,k} \), and
  \( \fs_0 \) by \( \fs_{-1} \), respectively. Then substitute
  Lemmas~\ref{lem:pm_parity} with \Cref{lem:shifted_pm_parity}. With
  these replacements, the desired conclusion follows.
\end{proof}

\begin{prop}\label{prop:bij_YZ_even}
  There exists a weight-preserving bijection between
  \( \mathcal{Y}^e_{a,b,k} \) and \( \mathcal{Z}^e_{a,b,k} \).
\end{prop}
\begin{proof}
  From the definitions of the subsets
  \( \mathcal{Y}^e_{a,b,k} \subseteq \mathcal{Y}_{2a,2b,k,0} \) and
  \( \mathcal{Z}^e_{a,b,k} \subseteq \mathcal{Z}_{2a,2b,k,0} \), every
  element of \( \mathcal{Y}^e_{a,b,k} \) and
  \( \mathcal{Z}^e_{a,b,k} \) has an even \( 0 \)th frequency, and
  each even part appears an even number of times. We now show that the
  bijection \( \phi \) constructed in \Cref{prop:bij_YZ} restricts to
  a bijection from \( \mathcal{Y}^e_{a,b,k} \) to
  \( \mathcal{Z}^e_{a,b,k} \).

  Unlike \Cref{prop:bij_YZ_odd}, because of the additional parity
  conditions defining \( \mathcal{Y}^e_{a,b,k} \) and
  \( \mathcal{Z}^e_{a,b,k} \), every element in these sets must have
  an even \( 0 \)th frequency. Hence, we need to verify explicitly
  that the map \( \phi \) preserves this property.

  If we take the variables \( j, r \), and the parameter \( \ell \) in
  \Cref{prop:bij_YZ} to be even, that is,
  \[
    j = 2a, \quad r = 2b, \quad \ell = 2m,
  \]
  then the restriction of \( \phi \) to \( \mathcal{Y}^e_{a,b,k} \)
  can be described explicitly as follows:
  \[
    (\dots, 0, f_0,f_1,f_2, \dots) \mapsto (\dots, 0, f_0', f_1, f_2,\dots),
  \]
  where if \( a \geq b \), then
  \[
    f_0' =
    \begin{cases}
      f_0 & \mbox{if \( f_0 \leq 2(a-b) \)},\\
      2(a-b+m) & \mbox{if \( f_0 = 2(a-b+2m) \) for some \( 2m \in \{1,\dots,2b\} \),}
    \end{cases}
  \]
  and if \( a < b \), then
  \[
    f_0' = 2m, \quad\text{for \( f_0 = 2(b-a+2m) \)
  with \( 2m \in \{0,\dots,2a\}\).}
  \]

  By the explicit description of \( \phi \), if \( f_0 \) is even then
  \( f_0' \) is even. Hence,
  \( \phi(\mathcal{Y}^e_{a,b,k}) \subseteq \mathcal{Z}^e_{a,b,k} \).
  Moreover, by the explicit formula again, it is easy to verify that
  the value \( f_0 \) obtained from an even \( f_0' \) always
  satisfies the defining condition of \( f_0 \) in
  \( \mathcal{Y}^e_{a,b,k} \). This shows
  \( \phi(\mathcal{Z}^e_{a,b,k}) \subseteq \mathcal{Y}^e_{a,b,k} \) and completes the proof.
\end{proof}

Finally, we use \Cref{thm:1} to express the generating function for \( \mathcal{Y}^o_{j,r,k} \) as a sum of products.

\begin{prop}\label{pro:2}
  For non-negative integers \( a \) and \( b \) with
  \( 2a + 2b \leq k \), we have
  \[
    \sum_{ f \in \mathcal{Y}^e_{a,b,k} } q^{|f|} =  
    \sum_{ s = 0 }^{ a } \frac{(-q;q^2)_\infty (q^{k+1+2a-2b-4s}, q^{k+3-2a+2b+4s}, q^{2k+4}; q^{2k+4})_\infty}{(q^2;q^2)_\infty}.
  \]
\end{prop}
\begin{proof}
  This proof is similar to that of \Cref{pro:gf Y}. From
  \Cref{thm:1}, the generating function for the frequency sequences in \( \mathcal{Y}^e_{a,b,k} \) satisfying \( f_0 = 2s \) is
\begin{equation}\label{eq:certain f0}
  \frac{(-q;q^2)_\infty (q^{k+1-2s}, q^{k+3+2s}, q^{2k+4}; q^{2k+4})_\infty}{(q^2;q^2)_\infty}.
\end{equation}

We use this to compute the generating function for \( \mathcal{Y}^e_{a,b,k} \).
If \( a \leq b \), then the set of all possible values for \( f_0 \) is
\[
  \{ 2(\ell+\max\{\ell-(a-b),0 \}) : 0 \leq \ell \leq a\} = \{ -2a + 2b + 4 \ell : 0 \leq \ell \leq a\},
\]
and the proof follows immediately from \eqref{eq:certain f0}. 

If \( a > b \), then the set
\[
 F= \{ 2(\ell+\max\{\ell-(a-b),0 \}) : 0 \leq \ell \leq a\} = \{ 0,2,4, \dots, 2a-2b-2, 2a-2b, 2a-2b+4, \dots, 2a+2b-4, 2a+2b\}
\]
can be expressed as the disjoint union $F= F_1 \sqcup F_2$, where
\[
F_1 =  \{2a+2b, 2a+2b-4, 2a+2b-8, \dots \}
\]
and
 \[
F_2 = \{ 2a-2b-2, 2a-2b-6, 2a-2b-10, \dots \}.
\]
The first set is given by
\[
 F_1= \{2b-2a+4s : s = \lfloor (a-b)/2 \rfloor, \lfloor (a-b)/2 \rfloor +1, \dots, a \},
\]
and the second by
\[
F_2=  \{ 2a - 2b - 2 -4s : s = 0,1, \dots, \lfloor (a-b)/2 \rfloor -1 \}.
\]
By~\eqref{eq:certain f0}, the corresponding generating functions are, respectively,
\[
\sum_{ f \in \mathcal{Y}^e_{a,b,k}: f_0 \in F_1} q^{|f|} =    \sum_{ s = \lfloor (a-b)/2 \rfloor }^{ a } \frac{(-q;q^2)_\infty (q^{k+1-(2b-2a+4s)}, q^{k+3+(2b-2a+4s)}, q^{2k+4}; q^{2k+4})_\infty}{(q^2;q^2)_\infty},
\]
and
\begin{multline*}
\sum_{ f \in \mathcal{Y}^e_{a,b,k}: f_0 \in F_2} q^{|f|} =   \sum_{ s = 0 }^{ \lfloor (a-b)/2 \rfloor -1 } \frac{(-q;q^2)_\infty (q^{k+1-(2a-2b-2-4s)}, q^{k+3+(2a-2b-2-4s)}, q^{2k+4}; q^{2k+4})_\infty}{(q^2;q^2)_\infty} =  \\
  \sum_{ s = 0 }^{ \lfloor (a-b)/2 \rfloor -1 } \frac{(-q;q^2)_\infty (q^{k+1+ 2a-2b-4s}, q^{k+3-2a+2b+4s}, q^{2k+4}; q^{2k+4})_\infty}{(q^2;q^2)_\infty}.
\end{multline*}
Adding the two expressions completes the proof.
\end{proof}

\section{Proof of \Cref{thm:even2} and Corollary \ref{cor:even}}
\label{sec:last}

The beginning of this section is very similar to the previous one, and so we give less detail here. However, the end of the proof of \Cref{thm:even2} requires a new argument, which we detail in Proposition \ref{prop:endtheven2}.

\begin{defn}\label{def:Xe'}
  Let \( a,b \), and \( k \) be non-negative integers with
  \( 2a + 2b -1 \leq k \). Define \( \widetilde{\mathcal{X}}^e_{a,b,k} \)
  to be the set of all pairs  \( (\bla, \fs_{-1}(\bla)) \), where
  \( \bla = (\lambda^{(1)},\dots,\lambda^{(k)}) \) is a
   \( k \)-tuple of partitions, subject to the conditions that for all $m \in \{1, \dots, k\}$,
  \begin{itemize}
  \item each part of \( \lambda^{(m)} \) is at least
    \( m + \max\{ m - 2a, 0\} + \max\{m-(k-2b+1), 0\} \) for each \( m = 1,\dots,k \), and
  \item \( \lambda^{(m)} \) is a partition into even parts.
  \end{itemize}
\end{defn}

We first compute the generating function for  \( \widetilde{\mathcal{X}}^e_{a,b,k} \).
\begin{prop}\label{pro:gf_Xe'}
Let \( a,b \), and \( k \) be non-negative integers with
  \( 2a + 2b -1 \leq k \). Then
  \[
    \sum_{(\bla,\fs_{-1}(\bla)) \in \widetilde{\mathcal{X}}_{a,b,k}^e} q^{|(\bla,\fs_{-1}(\bla))|} 
    = \sum_{ s_1 \geq \dots \geq s_k \geq 0 } \frac{q^{s_1^2 + \cdots + s_k^2 -2(s_2 + s_4 + \cdots + s_{2a}) +2(s_{k-2b+2} + s_{k-2b+4} + \cdots + s_{k})}}{(q^2;q^2)_{s_1 - s_2} \cdots (q^2;q^2)_{s_{k-1} - s_k}(q^2;q^2)_{s_k}}.
  \]
\end{prop}
\begin{proof}
  This proof follows the same argument as \Cref{pro:gf_Xe}, with the
  following modifications: \( \mathcal{X}_{a,b,k}^e \) and
  \( X_{a,b}^e(s_1, \dots, s_k) \) are replaced by
  \( \widetilde{\mathcal{X}}_{a,b,k}^e \) and
  \( \widetilde{X}_{a,b}^e(s_1, \dots, s_k) \), respectively, and the part leading
  to \eqref{eq:Xe_term2} and \eqref{eq:Xe_term3} is replaced by the
  following.
  
  For \( 2a < m \leq k-2b+1 \), parts of \( \lambda^{(m)} \) are at
  least \( 2m - 2a \), hence
  \begin{equation*}
    A_2 = \frac{q^{(2a+2)s_{2a+1}}q^{2s_{2a+2} + 2s_{2a+3} + \cdots + 2s_{k-2b+1}} q^{-(2k-2a-4b+2) s_{k-2b+2}}}{(q^2;q^2)_{s_{2a+1} - s_{2a+2}} \cdots (q^2;q^2)_{s_{k-2b+1} - s_{k-2b+2}}}.
  \end{equation*}
  For \( k-2b+1 < m \leq k \), parts of \( \lambda^{(m)} \) are at least
  \( 3m - k -2a + 2b -1 \), hence
  \begin{equation*}
    A_3 = \frac{q^{(2k-2a-4b+6)s_{k-2b+2}}q^{2s_{k-2b+3} + 4s_{k-2b+4} + 2s_{k-2b+5} + 4s_{k-2b+6} + \cdots + 2s_{k-1} + 4s_{k}}}{(q^2;q^2)_{s_{k-2b+1} - s_{k-2b+2}} \cdots (q^2;q^2)_{s_k}}.
  \end{equation*}

  This yields
  \[
    \sum_{ \bla \in \widetilde{X}_{a,b}^e(s_1, \dots, s_k) } q^{|\bla|} = \frac{q^{2(s_1 + s_3 + \cdots + s_{2a-1}) + 2(s_{2a+1} + s_{2a+2} + \cdots + s_{k-2b} + s_{k-2b+1}) + (4s_{k-2b+2} + 2s_{k-2b+3} + 4s_{k-2b+4} + \cdots + 2s_{k-1} + 4s_k)}}{(q^2;q^2)_{s_1 - s_2} \cdots (q^2;q^2)_{s_{k-1} - s_k} (q^2;q^2)_{s_k}}.
  \]
  Therefore, we obtain the desired formula.
\end{proof}

We show that  \( \widetilde{\mathcal{X}}^e_{a,b,k} \) and \( \widetilde{\mathcal{Z}}^e_{a,b,k} \) are in bijection.

\begin{prop}\label{prop:bij_XZ_even'}
  There exists a weight-preserving bijection between
  \( \widetilde{\mathcal{X}}^e_{a,b,k} \) and \( \widetilde{\mathcal{Z}}^e_{a,b,k} \).
\end{prop}
\begin{proof}
  This is exactly the same as the proof of \Cref{prop:bij_XZ_even} but
  with \( r = 2b-1 \) instead of \( r = 2b \).
  \end{proof}

Finally, we need one last proposition in order to use \Cref{thm:even1} to finish the proof of \Cref{thm:even2}.

\begin{prop}
\label{prop:endtheven2}
The following equality holds:
$$(1+q) \sum_{f\in \widetilde{\mathcal{Z}}_{a,b,k}^e}q^{|f|} = \sum_{f \in \mathcal{Z}_{a,b-1,k}^e}q^{|f|} + q\sum_{f\in \mathcal{Z}_{a,b,k}^e}q^{|f|}.$$
\end{prop}
\begin{proof}
We have the inclusions
\( \mathcal{Z}_{a,b,k}^e \subseteq \widetilde{\mathcal{Z}}_{a,b,k}^e
\subseteq \mathcal{Z}_{a,b-1,k}^e \). Moreover, the map
\( (f_0, f_1,f_2, \dots) \mapsto (f_0, f_1-1,f_2, \cdots ) \) is a
bijection from
\( \widetilde{\mathcal{Z}}_{a,b,k}^e \setminus \mathcal{Z}_{a,b,k}^e\)
to
\( \mathcal{Z}_{a,b-1,k}^e \setminus \widetilde{\mathcal{Z}}_{a,b,k}^e
\). Thus
\begin{align*}
  (1+q) \sum_{f \in \widetilde{\mathcal{Z}}_{a,b,k}^e}q^{|f|}
  &= \sum_{f \in \widetilde{\mathcal{Z}}_{a,b,k}^e}q^{|f|} + \sum_{f \in \widetilde{\mathcal{Z}}_{a,b,k}^e \setminus \mathcal{Z}_{a,b,k}^e}q^{|f|+1} + q\sum_{f \in \mathcal{Z}_{a,b,k}^e}q^{|f|}\\
  &= \sum_{f \in \widetilde{\mathcal{Z}}_{a,b,k}^e}q^{|f|} + \sum_{f \in \mathcal{Z}_{a,b-1,k}^e \setminus \widetilde{\mathcal{Z}}_{a,b,k}^e}q^{|f|} + q\sum_{f\in \mathcal{Z}_{a,b,k}^e}q^{|f|} \\
  &= \sum_{f \in \mathcal{Z}_{a,b-1,k}^e}q^{|f|} + q\sum_{f \in \mathcal{Z}_{a,b,k}^e}q^{|f|}.
\end{align*}
\end{proof}

Using \Cref{thm:even1} to evaluate the two sums in Proposition \ref{prop:endtheven2} completes the proof of \Cref{thm:even2}.

\medskip
We conclude with the proof of \Cref{cor:even}.
\begin{proof}[Proof of \Cref{cor:even}]
  By \Cref{thm:even1} and \Cref{thm:even2} with \( b = 0 \), we obtain
  \begin{equation}\label{eq:cor_even1}
 \sum_{ s_1 \geq \dots \geq s_k \geq 0 } \frac{q^{s_1^2 + \cdots + s_k^2 -2(s_2 + s_4 + \cdots + s_{2a})}}{(q^2;q^2)_{s_1 - s_2} \cdots (q^2;q^2)_{s_{k-1} - s_k}(q^2;q^2)_{s_k}} 
   = \sum_{ s = 0 }^{ a } \frac{(-q;q^2)_\infty (q^{k+1+2a-4s}, q^{k+3-2a+4s}, q^{2k+4}; q^{2k+4})_\infty}{(q^2;q^2)_\infty},
  \end{equation}
  and
  \begin{multline}\label{eq:cor_even2}
   \sum_{ s_1 \geq \dots \geq s_k \geq 0 } \frac{q^{s_1^2 + \cdots + s_k^2 -2(s_2 + s_4 + \cdots + s_{2a})}}{(q^2;q^2)_{s_1 - s_2} \cdots (q^2;q^2)_{s_{k-1} - s_k}(q^2;q^2)_{s_k}} \\
   = \frac{(-q^3;q^2)_\infty}{(q^2;q^2)_\infty}\sum_{ s = 0 }^{ a } \Bigg[ (q^{k+3+2a-4s}, q^{k+1-2a+4s}, q^{2k+4}; q^{2k+4})_\infty + q (q^{k+1+2a-4s}, q^{k+3-2a+4s}, q^{2k+4}; q^{2k+4})_\infty \Bigg],
  \end{multline}
  respectively. We first verify that, under this specialisation, the
  right-hand sides of \eqref{eq:cor_even1} and \eqref{eq:cor_even2}
  coincide. Indeed, in the right-hand side of \eqref{eq:cor_even2}, we
  apply the change of variables \( s \mapsto a-s \) in the second
  \( q \)-Pochhammer symbol appearing inside the summation. Under this
  reindexing, the second \( q \)-Pochhammer symbol is equal to the
  first one, and hence two right-hand sides are the same.

  We next show that the desired right-hand side can be obtained from
  the right-hand side of \eqref{eq:cor_even1}. The argument is similar
  to that used in the proof of~\Cref{cor:odd}. We split the sum in the
  right-hand side of \eqref{eq:cor_even1} into the two ranges
  \( s = 0, \dots, \lfloor a/2 \rfloor \) and
  \( s = \lfloor a/2 \rfloor +1, \dots, a \). For
  \( s = \lfloor a/2 \rfloor +1, \dots, a \), we have
  \( k+3-2a+4s < k+1 \). Thus, the right-hand side of
  \eqref{eq:cor_even1} can be rewritten as
  \begin{multline}\label{eq:cor_even3}
    \sum_{ s = 0 }^{ \lfloor a/2 \rfloor } \frac{(-q;q^2)_\infty (q^{k+1+2a-4s}, q^{k+3-2a+4s}, q^{2k+4}; q^{2k+4})_\infty}{(q^2;q^2)_\infty}\\
    + \sum_{ s = \lfloor a/2 \rfloor +1}^{ a } \frac{(-q;q^2)_\infty (q^{k+3-2a+4s}, q^{k+1+2a-4s}, q^{2k+4}; q^{2k+4})_\infty}{(q^2;q^2)_\infty}.
  \end{multline}

  Note that the set \( \{k+1-2i : i = 0, \dots, a\} \) can be decomposed
  as a disjoint union
  \begin{multline*}
    (\{k+1-2a, k+1-2a+4, k+1-2a+8, \dots \} \cap \{1, \dots, k+1\}) \\
    \sqcup  (\{k+3-2a, k+3-2a+4, k+3-2a+8, \dots \} \cap \{1, \dots, k+1\}),
  \end{multline*}
  and the first and second sets are equal to
  \[
    \{k+1+2a-4s: s = \lfloor a/2 \rfloor+1, \dots, a \},\qand \{k+3-2a+4s: s = 0, \dots, \lfloor a/2 \rfloor\},
  \]
  respectively. Combining this decomposition with \eqref{eq:cor_even3}
  yields the desired result.
\end{proof}

\section{Open problems}
\label{sec:problems}
We propose two directions for further research.

First, one may seek combinatorial proofs of the identities that follow
directly from our results. From \Cref{thm:main}, we obtain two
different infinite summation expressions of the same product
expression. More precisely, specializing \Cref{thm:main} with
\( j=k-a \) and \( r=a \), we obtain \( \operatorname{AK}_{a,k}(q) \)
appearing in \Cref{cor:2}. Taking the difference
\( \operatorname{AK}_{a,k}(q) - \operatorname{AK}_{a+1,k}(q) \) yields
\[
  \sum_{s_1 \ge \cdots \ge s_k \geq 0} \frac{q^{s_1^2 + \cdots + s_k^2 - s_1 - \cdots - s_{k-a} + s_{k-a+1} + \cdots + s_k}(1-q^{2s_{k-a}})}{(q^2;q^2)_{s_1 - s_2} \cdots (q^2;q^2)_{s_{k-1} - s_k}(q^2;q^2)_{s_k}}
  = \frac{(-q^2;q^2)_\infty(q^{2a+2}, q^{2k+2-2a}, q^{2k+4}; q^{2k+4})_\infty}{(q^2;q^2)_\infty}.
\]
On the other hand, specializing \Cref{thm:main} with \( j=0 \) and
\( r = k-2a \) gives another identity
\[
  \sum_{s_1 \ge \cdots \ge s_{k} \geq 0}
  \frac{q^{s_1^2 + \cdots + s_{k}^2 + (s_1 - s_2 + s_3 - \cdots + s_{2a-1} - s_{2a}) + (s_{2a+1} + \cdots + s_{k})}}
  {(q^2;q^2)_{s_1 - s_2} \cdots (q^2;q^2)_{s_{k-1} - s_{k}}(q^2;q^2)_{s_{k}}}
  = \frac{(-q^2;q^2)_\infty(q^{2a+2}, q^{2k+2-2a}, q^{2k+4}; q^{2k+4})_\infty}{(q^2;q^2)_\infty}.
\]
By the two identities above, we obtain
\begin{multline}\label{eq:open1}
  \sum_{s_1 \ge \cdots \ge s_k \geq 0} \frac{q^{s_1^2 + \cdots + s_k^2 - s_1 - \cdots - s_{k-a} + s_{k-a+1} + \cdots + s_k}(1-q^{2s_{k-a}})}{(q^2;q^2)_{s_1 - s_2} \cdots (q^2;q^2)_{s_{k-1} - s_k}(q^2;q^2)_{s_k}} \\
  =
  \sum_{s_1 \ge \cdots \ge s_{k} \geq 0}
  \frac{q^{s_1^2 + \cdots + s_{k}^2 + (s_1 - s_2 + s_3 - \cdots + s_{2a-1} - s_{2a}) + (s_{2a+1} + \cdots + s_{k})}}
  {(q^2;q^2)_{s_1 - s_2} \cdots (q^2;q^2)_{s_{k-1} - s_{k}}(q^2;q^2)_{s_{k}}}.
\end{multline}

\begin{problem}
  Find a combinatorial proof of \eqref{eq:open1}.
\end{problem}

Similarly, comparing product expressions leads to identities between
two summation expressions having different numbers of summation
variables. By \cite[Theorem~1.12]{DJJ25} and \Cref{thm:even2}, we
obtain
\begin{multline}\label{eq:open2}
  \sum_{s_1\geq\dots\geq s_{k}\geq0}\frac{q^{2(s_1^2+\dots+s_{k}^2-s_1-\dots-s_{j}+s_{k-r+1}+\dots+s_{k})}(-q^{1+2s_k};q^2)_\infty}{(q^2;q^2)_{s_1-s_2}\cdots(q^2;q^2)_{s_{k-1}-s_{k}}(q^2;q^2)_{s_k}}  = \\
  \sum_{ s_1 \geq \dots \geq s_{2k} \geq 0 } \frac{q^{s_1^2 + \cdots + s_{2k}^2 -2(s_2 + s_4 + \cdots + s_{2j}) +2(s_{2k-2r+2} + s_{2k-2r+4} + \cdots + s_{2k})}}{(q^2;q^2)_{s_1 - s_2} \cdots (q^2;q^2)_{s_{2k-1} - s_{2k}}(q^2;q^2)_{s_{2k}}}.
\end{multline}
\begin{problem}
  Find a combinatorial proof of \eqref{eq:open2} (for instance, via
  particle motion, or by any other combinatorial method).
\end{problem}

Second, we propose generalisations in other directions and alternative
proofs of our results. Stanton~\cite{Stanton2018} proved not only the
identity~\eqref{eq:sAG}, but also a binomial
extension~\cite[Theorem~3.1]{Stanton2018} of the Andrews--Gordon
identities: Let \( j,r \geq 0 \) and \( k \geq 1 \) be integers such
that \( j+r \leq k \). Then
\begin{multline}\label{eq:open4}
  \sum_{s_1 \geq \cdots \geq s_k \geq 0}q^{s_1^2 + \cdots + s_k^2 + s_{k-r+1} + \cdots + s_k} \cdot \frac{q^{-s_1 - \cdots - s_j} (1 + q^{s_1 + s_2})(1 + q^{s_2 + s_3}) \cdots (1 + q^{s_{j-1} + s_j})}{(q)_{s_1 - s_2} \cdots (q)_{s_{k-1} - s_k} (q)_{s_k}} \\
  = \sum_{s=0}^j \binom{j}{s} \frac{(q^{2k+3},q^{k+1-r+j-2s}, q^{k+2+r-j+2s} ; q^{2k+3})_\infty}{(q)_\infty}.
\end{multline}
Moreover, the \(j\) factors \(q^{-s_1}\) and
\(q^{-s_i}(1+q^{s_{i-1}+s_i})\), \(2 \leq i \leq j\) may be replaced
by any \(j\)-element subset of
\(\{q^{-s_1}\} \cup \{q^{-s_i}(1+q^{s_{i-1}+s_i}) : 2 \leq i \leq
k-r\}\). It would be interesting to develop analogous binomial
extensions of our results.
\begin{problem}
  Find binomial extensions of
  \Cref{thm:main},~\Cref{thm:even1},~and~\Cref{thm:even2} in the
  spirit of \eqref{eq:open4}.
\end{problem}

In joint work~\cite{DJJ25} of the authors with Jouhet, Stanton-type
generalisations with binomial coefficients were proved using Bailey
pairs. If explicit formulas answering the previous problem can be
found, it would be natural to seek Bailey pair proofs of those
identities. Likewise, it would be interesting to obtain Bailey pair
proofs of our main results.

\begin{problem}
  Prove \Cref{thm:main},~\Cref{thm:even1},~and~\Cref{thm:even2} using
  Bailey pairs.
\end{problem}

\section*{Acknowledgements}

The authors are supported by the SNSF Eccellenza grant PCEFP2 202784. They are grateful to Frédéric Jouhet for asking whether is is possible to prove the Chern--Li--Stanton--Xue--Yee identity using particle motion and for his comments on an earlier version of this paper.

\bibliographystyle{alpha}
%\bibliography{/Users/jihyeug/Desktop/newbiemacs/nbm-user-settings/references/ref.bib}

\end{document}